\documentclass[12pt]{amsart}
\usepackage[left=1in,top=1in,right=1in,bottom=1in,letterpaper]{geometry}
\usepackage[all]{xy}
\usepackage{amsmath}
\usepackage{amssymb}
\usepackage{amscd}
\usepackage{bm}

\usepackage{url}
\usepackage{verbatim} 
\usepackage{xcolor}
\usepackage{amsthm}
\usepackage{pifont}
\usepackage{DotArrow}
\usepackage{graphicx}
\usepackage{tikz-cd}
\usepackage{libertine}
\usepackage{hyperref}

\title[An infinite surface with the lattice property II]{An infinite surface with the lattice property II:\\ Dynamics of pseudo-Anosovs}

\author[]{W. Patrick Hooper}
\address{The City College of New York, New York, NY, USA 10031}
\address{CUNY Graduate Center, New York, NY, USA 10016}
\email{whooper@ccny.cuny.edu}

\ifdefined\theorem
\else
\newtheorem{theorem}{Theorem}
\newtheorem{proposition}[theorem]{Proposition}
\newtheorem{lemma}[theorem]{Lemma}

\newtheorem{corollary}[theorem]{Corollary}

\theoremstyle{definition}

\fi
\ifdefined\openquestion
\else
\newtheorem{openquestion}[theorem]{Open Question}
\fi

\ifdefined\question
\else
\newtheorem{question}[theorem]{Question}
\fi

\newlength{\savearraycolsep}
	{\setlength{\savearraycolsep}{\arraycolsep}%
	\setlength{\arraycolsep}{#1}%
	\begin{array}{#2}}%
	{\end{array}\setlength{\arraycolsep}{\savearraycolsep}}

\def\C{\mathbb{C}}%
\def\N{\mathbb{N}}%
\def\R{\mathbb{R}}%
\def\Z{\mathbb{Z}}%
\def\H{\mathbb{H}} 

\def\GL{\textit{GL}}
\def\SL{\textit{SL}}

\def\PGL{\textit{PGL}}

\def\0{{\mathbf{0}}}
\def\1{{\mathbf{1}}}

\def\ba{{\mathbf{a}}}

\def\bb{{\mathbf{b}}}
\def\bc{{\mathbf{c}}}
\def\bd{{\mathbf{d}}}

\def\bu{{\mathbf{u}}}
\def\u{{\mathbf{u}}}
\def\bv{{\mathbf{v}}} %
\def\v{{\mathbf{v}}} %
\def\bw{{\mathbf{w}}}

\def\bP{{\mathbf{P}}}

\def\bS{{\mathbf{S}}}

\def\sA{{\mathcal{A}}}
\def\sB{{\mathcal{B}}}
\def\sC{{\mathcal{C}}}

\def\sG{{\mathcal{G}}}
\def\sH{{\mathcal{H}}}

\def\sL{{\mathcal{L}}}

\def\sP{{\mathcal{P}}}

\def\sS{{\mathcal{S}}}
\def\sT{{\mathcal{T}}}

\def\sW{{\mathcal{W}}}

\def\del{{\partial}} 
\def\ker{\textit{ker}}%

\def\hol{\mathit{hol}} 

\makeatletter
\def\imod#1{\allowbreak\mkern10mu({\operator@font mod}\,\,#1)}
\makeatother

\def\dev{{\mathit{dev}}} 
\def\Aff{\mathit{Aff}} 

\def\and{{\quad \textrm{and} \quad}}

\newif\ifdraft\drafttrue

\def\Area{\textrm{Area}}

\newcommand{\bX}[0]{{\mathbf{X}}}

\newcommand{\Hom}[0]{{\mathrm{Hom}}}

\newcommand{\bbP}[0]{\mathbb{P}}

\def\KH{{\mathbb K\mathbb H}}

\begin{document}
\clearpage
\begin{abstract}
We study the behavior of hyperbolic affine automorphisms of a translation surface which is infinite in area and genus that is obtained
as a limit of surfaces built from regular polygons studied by Veech. We find that hyperbolic affine automorphisms are not recurrent and yet their action
restricted to cylinders satisfies a mixing-type formula with polynomial decay. Then we consider the extent to which the
action of these hyperbolic affine automorphisms satisfy Thurston's definition of a pseudo-Anosov homeomorphism.
In particular we study the action of these automorphisms on simple closed curves and on homology classes. 
These objects are exponentially attracted by the expanding and contracting foliations but exhibit polynomial decay. 
We are able to work out exact asymptotics of these limiting quantities because of special integral formula for
algebraic intersection number which is attuned to the geometry of the surface and its deformations.
\end{abstract}
\maketitle
\thispagestyle{empty}

\section*{Introduction}

Translation surfaces built from two copies of a regular polygon as depicted below were studied by Veech and proven to have beautiful properties \cite{V}. Perhaps most surprisingly, these surfaces
admit affine symmetries distinct from the obvious Euclidean symmetries. An understanding of these symmetries allowed Veech to prove his famed dichotomy theorem: In all but countably many directions every trajectory equidistributes, and the countably many exceptional directions are completely periodic. Veech also used this symmetry group to answer natural counting problems on these surfaces.

\begin{figure}
\begin{center}
\includegraphics[height=2.5in]{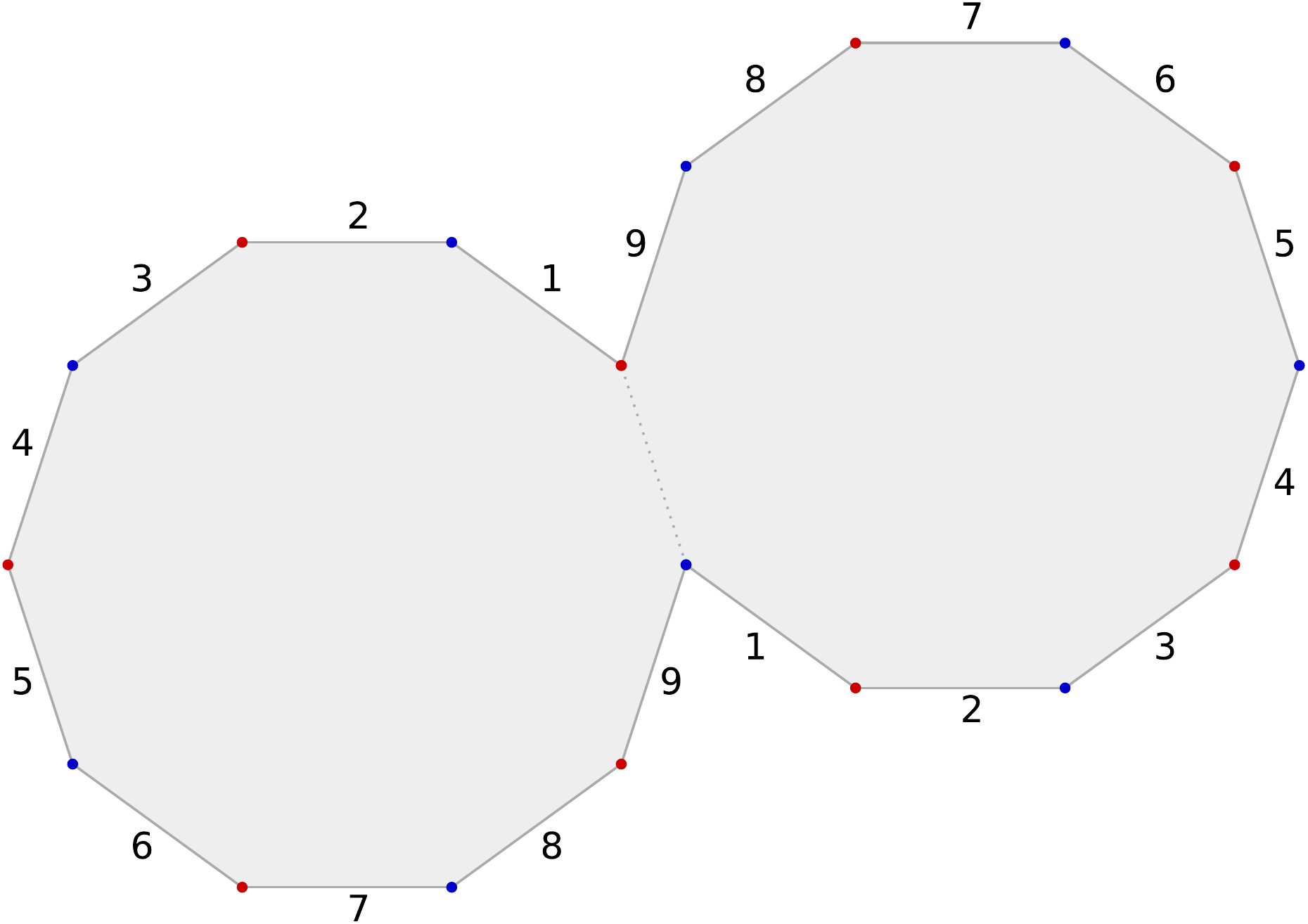}
\end{center}
\caption{Veech's double $10$-gon surface. Edge labels indicate glued edges.}
\end{figure}

Let $\bX_{n}$ denote the double regular $n$-gon surface. 
In \cite{Higl1}, we showed that by choosing affine maps $A_n$ of the plane which send three consecutive vertices of the regular $n$-gon to the points $(-1,1)$, $(0,0)$ and $(1,1)$, the sequence of surfaces $\bP_{\cos \frac{\pi}{n}} = A_n(\bX_n)$ converges to the infinite area surface $\bP_1$ built from two polygonal parabolas, the convex hulls of the sets $\{(n,n^2):~n \in \Z\}$ and $\{(n,-n^2):~n \in \Z\}$. The limiting surface $\bP_1$ is depicted in the center of Figure \ref{fig:s1}. 

\begin{figure}
\includegraphics[height=4in]{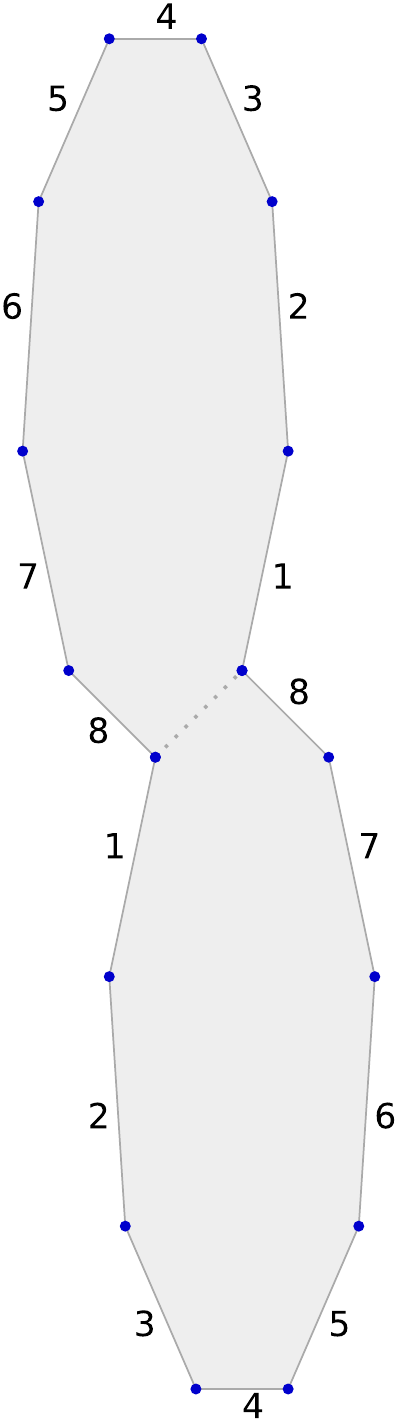}
\hfill
\includegraphics[height=4in]{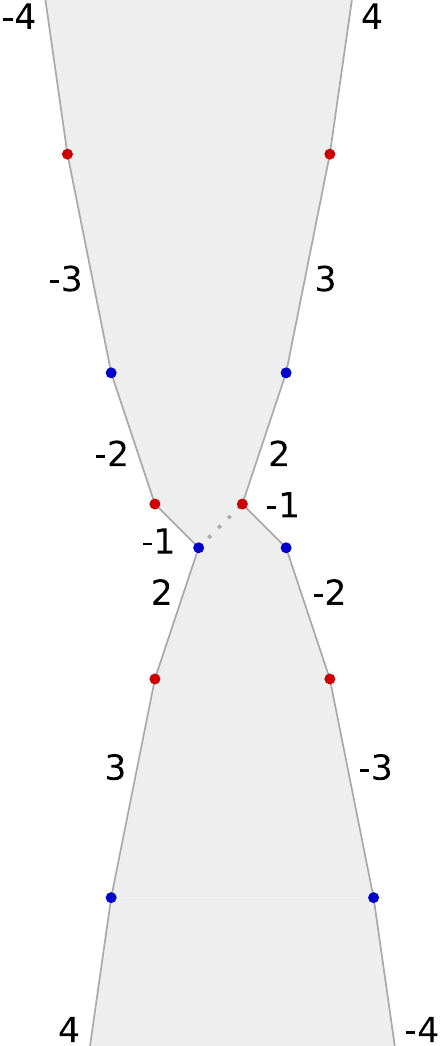}
\hfill
\includegraphics[height=4in]{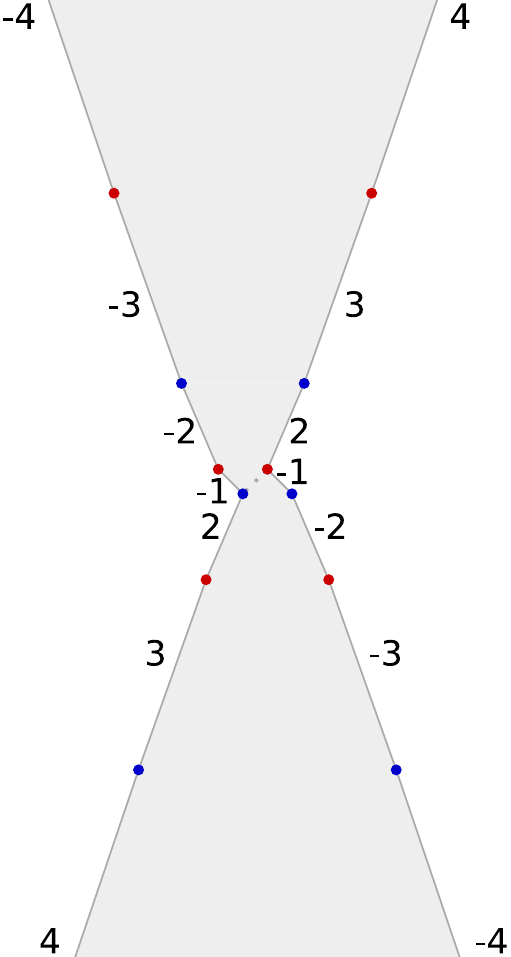}
\caption{From left to right, the surfaces $\bP_{\cos \frac{2 \pi}{9}}$, $\bP_1$ and $\bP_{\frac{5}{4}}$ are shown.}
\label{fig:s1}
\end{figure}

An {\em affine automorphism} of a translation surface $\bS$ is a homeomorphism $\phi:\bS \to \bS$ which is a real affine map in local coordinates. Using the natural identification
between tangent spaces to non-singular points of $\bS$ with the plane, we see that the derivative $D \phi$ of such a map must be constant and so we interpret $D\phi$ as an element
of $\GL(2,\R)$. The subgroup of $\GL(2,\R)$ consisting of derivatives of affine homeomorphisms of $\bS$ is called the surface's {\em Veech group}. (We allow orientation reversing elements in the Veech group, which differs from conventions in some other articles.)
For the surfaces we will consider, the Veech groups are contained in $\SL^\pm(2,\R)$,
the group of $2 \times 2$ matrices with real entries and determinant $\pm 1$.

In \cite{Higl1}, it was shown that the symmetries of $\bP_{\cos \frac{\pi}{n}}$ persist to the limiting surface $\bP_1$ and beyond: There are surfaces $\bP_c$ defined for $c \geq 1$ and all these surfaces have topologically conjugate affine automorphism group actions \cite[Theorem 7]{Higl1}. The Veech groups
of $\bP_c$ vary continuously in $c$ and lie in the image of the representation $\rho_c:\sG \to \SL^\pm(2,\R)$ where $\sG=(C_2 \ast C_2 \ast C_2) \times C_2$ with $C_2$ denoting the cyclic group of order two and
$$\rho_c(\sG)=\left \langle
\left[\begin{array}{cc}
-1 & 0 \\
0 & 1
\end{array}\right],
\left[\begin{array}{cc}
-1 & 2 \\
0 & 1
\end{array}\right],
\left[\begin{array}{cc}
-c & c-1 \\
-c-1 & c
\end{array}\right],
\left[\begin{array}{cc}
-1 & 0 \\
0 & -1
\end{array}\right]
\right \rangle$$
describes the images of the generators.

We call a matrix in $\SL^\pm(2,\R)$ {\em hyperbolic} if it has two eigenvalues with distinct absolute values
and call an affine automorphism {\em hyperbolic} if its derivative is hyperbolic. 
A hyperbolic matrix in $\SL^\pm(2,\R)$ has two real eigenvalues $\lambda^u$ and $\lambda^s$ with $|\lambda^u|>1$, $|\lambda^s|<1$ and $\lambda^u \lambda^s = \pm 1$.
The Veech group $\rho_1(\sG)$ of $\bP_1$ has numerous hyperbolic elements. 

A significant goal of this paper is to address the extent to which hyperbolic affine automorphisms of infinite surfaces satisfy the defining properties of pseudo-Anosov homeomorphisms of closed surfaces. The example we study is very special which enables us 
to say more about these questions than we'd expect to be able to for a general surface, so in particular we expect this paper
sets some limits on what we could hope to be true in general.

For closed surfaces, the dynamics of hyperbolic affine automorphisms are well known to be mixing since they admit Markov partitions \cite[\S 10.5]{FLP}. On the surface $\bP_1$ we observe that cylinders satisfy a mixing-type formula but with polynomial decay. A {\em cylinder} on a translation surface is a subset isometric to $\R/c \Z \times [0,h]$ for some circumference $c>0$ and height $h>0$.
We say two sequences $a_n$ and $b_n$ are asymptotic and write $a_n \sim b_n$ if $\lim_{n \to \infty} \frac{a_n}{b_n}=1$.

\begin{theorem}
\label{thm:mixing}
Let ${\mathcal A}$ and ${\mathcal B}$ be cylinders in $\bP_1$ and $\phi:\bP_1 \to \bP_1$ be a hyperbolic affine automorphism
with derivative $D \phi=\rho_1(g)$. Then 
$$\Area \big(\phi^n(\sA) \cap \sB\big) \sim
{\textstyle \frac{1}{4 \sqrt{2 \pi}} \big(\frac{1}{\beta n}\big)^{\frac{3}{2}}} \Area(\sA) \Area(\sB) 
\quad
\text{where} \quad
\beta=\frac{1}{\lambda_1^u}\, [{\textstyle \frac{d}{dc}} \lambda_c^u]_{c=1}$$
is a positive constant which can be computed using the formula above where we use
$\lambda_c^u$ to denotes the eigenvalue of $\rho_c(g)$ with greatest absolute value.
\end{theorem}

Note that $\lambda^u_c$ is real and varies analytically in a neighborhood of $c=1$
because the entries of the matrix $\rho_c(g)$ are real polynomials in $c$ and $\rho_1(g)$ is hyperbolic (and hyperbolicity is stable under perturbation of the matrix). This ensures that the quantity $[{\textstyle \frac{d}{dc}} \lambda_c^u]_{c=1}$ is a well-defined real number.
It will follow from later work (Lemma \ref{lem:eigenvalue}) that the quantity $\beta$ is positive.

As a consequence of this theorem, we note that no hyperbolic $\phi$ is recurrent because of
the $n^{\frac{-3}{2}}$ decay rate seen above:

\begin{corollary}
\label{cor: dissipative}
If $\phi:\bP_1 \to \bP_1$ is hyperbolic then its action on $\bP_1$ is totally dissipative: 
there is a countable collection $\sW$ of Lebesgue-measurable subsets of $\bP_1$ so that
\begin{enumerate}
\item the Lebesgue measure of the complement $\bP_1 \smallsetminus \bigcup_{W \in \sW} W$ is zero, and 
\item each $W \in \sW$ is wandering in the sense that the collection $\{\phi^{-n}(W):~n \geq 0\}$ is pairwise disjoint.
\end{enumerate}
\end{corollary}

\begin{figure}
\includegraphics[width=4in]{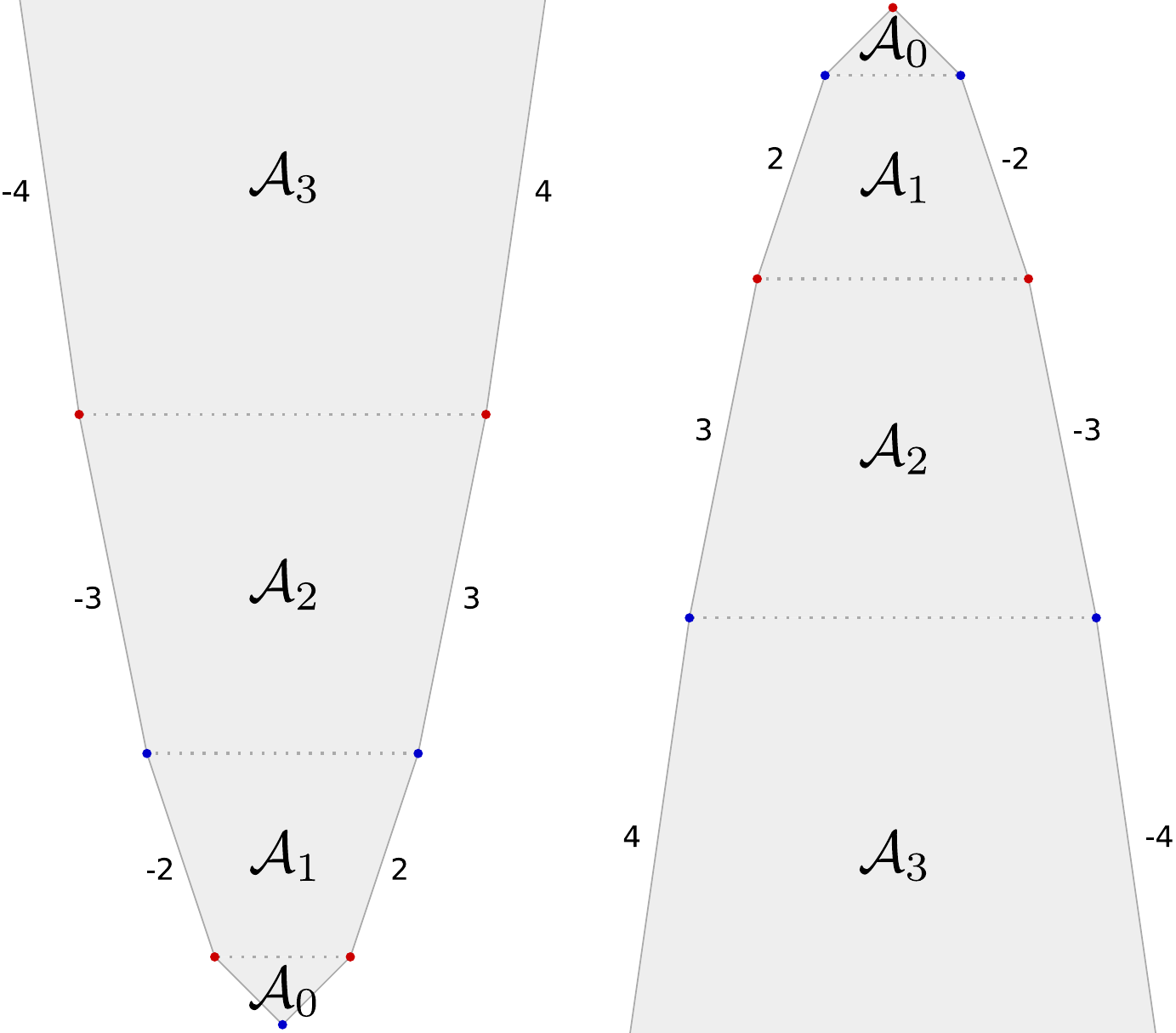}
\caption{The horizontal cylinders $\sA_i$ of $\bP_1$.}
\label{fig:horizontal}
\end{figure}

Because of non-recurrence, we tend to think points as the wrong objects to act to study hyperbolic affine automorphisms of $\bP_1$.
Fortunately, some famous observations of Thurston suggest that acting on simple closed curves might be more natural.

We briefly recall Thurston's definition of a pseudo-Anosov homeomorphism of a closed surface $M$ following \cite{Th86}.
Let $\sS = \sS(M)$ be the collection of homotopically non-trivial simple closed curves up to isotopy on $M$. Letting $i:\sS \times \sS \to \R$ denote geometric intersection number
(i.e., the minimum number of transverse intersections among curves from the isotopy classes), we have an induced map 
$$i_\ast:\sS \to \R^\sS \quad \text{defined by} \quad i_\ast(\alpha)(\beta)=i(\alpha,\beta)$$
so that the image is contained in the non-negative cone of $\R^\sS$. We let $\bbP \R^\sS$ denote the projectivization $\R^\sS/(\R \smallsetminus \{\0\})$
and $P:\R^\sS \to \bbP \R^\sS$ denote the projectivization map.
We endow $\R^\sS$ with the product topology and $\bbP \R^\sS$ with the quotient topology.
Thurston observed that for a compact surface $M$ the projectivized image $P \circ i_\ast\big(\sS(M)\big)$ has compact closure
which we will denote by $\sP \sS(M)$. As long as the surface $M$ has negative Euler characteristic,
the closure $\sP\sS(M)$ is a sphere of dimension one less than the dimension of the Teichm\"uller space $\sT(M)$ (and forms its {\em Thurston boundary}).
The space $\sP \sS(M)$ was identified with the space of projective measured foliations on $M$ which Thurston also introduced and gives a geometric meaning
to $\sP \sS$.
Homeomorphisms of $M$ naturally act on $\sP\sS(M)$ and isotopic homeomorphisms act in the same way. 
A {\em pseudo-Anosov} homeomorphism of $M$ is a homeomorphism $\phi:M \to M$ for which 
there are non-zero $\mu^u$ and $\mu^s$ in $\R^\sS$ such that their projections lie in $\sP \sS(M)$
and so that there is a
$\lambda>0$ so that $\mu^u \circ \phi^{-1}=\lambda \mu^u$ and $\mu^s \circ \phi=\lambda^{-1} \mu^s$. The action of a pseudo-Anosov homeomorphism $\phi$ on $\sP \sS(M)$ is analogous the action of a hyperbolic isometry on the boundary of hyperbolic space: 
For any $\alpha \in \sS(M)$,
\begin{equation}
\label{eq:pA curves}
\lim_{n \to +\infty} P \circ i_\ast\big(\phi^n(\alpha)\big)=P(\mu^u) 
\quad \text{and} \quad
\lim_{n \to +\infty} P \circ i_\ast\big(\phi^{-n}(\alpha)\big)=P(\mu^s),
\end{equation}
see \cite[Corollary 12.3]{FLP}. Thurston showed that more generally for any $P(\nu) \in \sP \sS(M) \smallsetminus \{P(\mu^u), P(\mu^s)\}$,
\begin{equation}
\label{eq:pA PS}
\lim_{n \to \infty} P(\nu \circ \phi^{-n})=P(\mu^u) 
\quad \text{and} \quad 
\lim_{n \to \infty} P(\nu \circ \phi^{n})=P(\mu^s).
\end{equation}

This paper investigates the extent to which the above results hold for the surface $\bP_1$.
We begin with trying to emulate the above definitions for $\bP_1$.

Perhaps we have been slightly abusing terminology to call $\bP_1$ a ``surface''. It has two infinite cone singularities coming from the vertices of the polygonal parabola. 
These singularities do not have neighborhoods locally homeomorphic to an open subset of the plane, so we define $\bP_1^\circ$ to be $\bP_1$ with these singularities removed.
The space $\bP_1^\circ$ is an infinite genus topological surface. We define $\sS = \sS(\bP_1^\circ)$ to be the collection of simple closed curves in $\bP_1^\circ$ up to isotopy and 
define $i_\ast:\sS \to \R^\sS$, $\bbP \R^\sS$ and $P:\R^\sS \to \bbP \R^\sS$ as above. We define
$$\sP \sS(\bP_1^\circ)=\overline{P \circ i_\ast\big(\sS(\bP_1^\circ)\big)}$$
only this time we note that $\sP \sS$ is not compact. (If $\alpha_n$ is a sequence of simple closed curves exiting every compact subset of $\bP_1^\circ$ then $
\lim i_\ast(\alpha_n)=\0 \in \R^\sS$ and no subsequence of $P \circ i_\ast(\alpha_n)$ converges in $\bbP \R^\sS$.) We show:

\begin{theorem}
\label{thm:geometric}
Fix a hyperbolic affine automorphism $\phi:\bP_1 \to \bP_1$ and let $\lambda^u_1 \in \R$ denote the expanding eigenvalue of $D \phi$.
Let $\mu^u$ and $\mu^s$ be the elements of $\R^{\sS(\bP_1^\circ)}$ corresponding to the transverse measures on $\bP_1$ to foliations parallel to the expanding and contracting 
eigenspaces of $D \phi$ respectively. Then:
\begin{enumerate}
\item We have $\mu^u \circ \phi^{-1}=|\lambda^u_1| \mu^u$ and $\mu^s \circ \phi^{-1}=|\lambda^u_1|^{-1} \mu^s$.
\item For any $\alpha \in \sS(\bP_1^\circ)$, equation \eqref{eq:pA curves} holds. In fact, in the space $\R^{\sS(\bP^\circ_1)}$ we have
$$\lim_{n \to \infty} {\textstyle \frac{n^{\frac{3}{2}}}{|\lambda^u_1|^n}} i_\ast\big(\phi^{n}(\alpha)\big)= \frac{\mu^s(\alpha)}{4 \beta^{\frac{3}{2}} \sqrt{2 \pi}\,|\bu^u \wedge \bu^s|} \mu^u
\quad \text{and} \quad 
\lim_{n \to \infty} {\textstyle \frac{n^{\frac{3}{2}}}{|\lambda^u_1|^n}} i_\ast\big(\phi^{-n}(\alpha)\big)= \frac{\mu^u(\alpha)}{4 \beta^{\frac{3}{2}} \sqrt{2 \pi}\,|\bu^u \wedge \bu^s|} \mu^s
$$
where $\beta$ is given as in Theorem \ref{thm:mixing}
and $\bu^u$ and $\bu^s$ denote expanding and contracting unit eigenvectors of $D \phi$.
In particular, we have $P \circ i_\ast\big(\phi^n(\alpha)\big) \to P(\mu^u)$
and $P \circ i_\ast\big(\phi^{-n}(\alpha)\big) \to P(\mu^s)$
so that
$P(\mu^u), P(\mu^s) \in \sP \sS(\bP_1^\circ)$.
\end{enumerate}
\end{theorem}

On the other hand, it is not true \eqref{eq:pA PS} holds for every $P(\nu) \in \sP \sS(\bP_1^\circ) \smallsetminus \{P(\mu^u), P(\mu^s)\}$.
For every direction $\theta$ of irrational slope and every $c>1$, there is a direction $\theta'$ and a homeomorphism $h:\bP_1 \to \bP_c$ (in the isotopy class of a standard identification between these surfaces) which carries
the foliation on $\bP_1$ in direction $\theta$ to the foliation in direction $\theta'$ on $\bP_c$. Furthermore, if $\theta$ was stabilized by a hyperbolic affine automorphism
$\phi:\bP_1 \to \bP_1$ then $\theta'$ is stabilized by an affine automorphism $\phi':\bP_c \to \bP_c$ which is the same up to the canonical identification and isotopy.
We can pull back the transverse measure on $\bP_c$ in direction $\theta'$ to obtain another measured on the foliation of $\bP_1$ in direction $\theta$ which is stabilized by $\phi$. This was proved in a more general context in \cite[Theorem 4.4]{Hinf}. This new measure corresponds to a distinct $\phi$-invariant element of $\R^{\sS(\bP^\circ_1)}$.
In fact we can see this element lives in $\sP \sS(\bP_1^\circ)$ in many cases because the straight-line flow in many eigendirections of pseudo-Anosov homeomorphisms is ergodic \cite[Theorem 4.5]{Hinf}, so that we can obtain projective approximations by simple closed curves by flowing point forward until it returns close and closing it up.
As we increase the flow time, convergence to $P(\mu^u)$ is guaranteed for almost every starting point by the ratio ergodic theorem.

We have shown that elements in $P(\sS(\bP_1^\circ))$ are attracted under $\phi$ and $\phi^{-1}$-orbits respectively by $P(\mu^u)$ and $P(\mu^s)$, but noted that this does not hold on the closure $\sP \sS(P_1^\circ)$, so it is natural to wonder how attractive $P(\mu^u)$ and $P(\mu^s)$ are in other contexts.

Since $\bP_1$ is a translation surface, it is natural to orient our foliations in each direction and to consider our transverse measures to be signed measures. 
Informally, let $H_1(\bP_1^\circ; \R)$ be real weighted finite sums of homology classes of closed curves in $\bP_1^\circ$, and let $H_1(\bP_1, \Sigma; \R)$ be real weighted finite sums of homology classes of closed curves and curves joining the singularities of $\bP_1$. Algebraic intersection number gives
a weakly non-degenerate bilinear map
$$\cap:H_1(\bP_1^\circ; \R) \times H_1(\bP_1, \Sigma; \R) \to \R.$$
Let $\sH=\Hom\big(H_1(\bP_1, \Sigma; \R), \R\big)$ be the collection of linear maps from
$H_1(\bP_1, \Sigma; \R)$ to $\R$ and let $\sP \sH=\sH/(\R \smallsetminus \{0\})$ be the projectivization of this space. In a parallel construction to geometric intersection number, we get a map induced by algebraic intersection and a projectivization map:
$$\cap_\ast:H_1(\bP_1^\circ; \R) \to \sH, \qquad P:\sH \to \sP\sH.$$

Given an oriented arc on $\bP_1$ we can lift it to the universal cover and project it to the plane under the developing map. The {\em holonomy vector} of the arc is the difference of the developed end point and the starting point. Holonomy gives rise to linear maps
$$\hol_1: H_1(\bP_1^\circ; \R) \to \R^2 
\quad \text{and} \quad 
\hol_1: H_1(\bP_1, \Sigma; \R) \to \R^2,$$
where we write $\hol_1$ to indicate we are computing holonomy on $\bP_1$.
Given a direction described by a unit vector $\bu \in \R^2$ we get an element of $\sH$ using the linear map
\begin{equation}
\label{eq:dual class to foliation}
H_1(\bP_1, \Sigma; \R) \to \R; \quad \sigma \mapsto \bu \wedge \hol_1(\sigma)
\end{equation}
where $\wedge: \R^2 \times \R^2 \to \R$ is the usual wedge product:  $(a,b) \wedge (c,d)=ad-bc$.

If $\phi:\bP_1 \to \bP_1$ is an affine homeomorphism with hyperbolic derivative, then the choice of unit unstable and stable eigenvectors
of $D \phi$ give rise via \eqref{eq:dual class to foliation} to respective elements $\mu^u$ and $\mu^s$ of 
$\sH$ satisfying
$$\mu^u \circ \phi^{-1}=\lambda^u_1 \mu^u
\quad \text{and} \quad
\mu^s \circ \phi^{-1}=\lambda^s_1 \mu^s$$
where $\lambda^u_1$ and $\lambda^s_1$ are the expanding and contracting eigenvalues of $D \phi$.
We show that the projectivized classes $P(\mu^u)$ and $P(\mu^s)$ respectively attract and repel every element of
$P \circ \cap_\ast\big(H_1(\bP_1^\circ; \R)\big)$,
but unlike in prior results the rate of polynomial decay depends on the chosen homology class.

\begin{theorem}
\label{thm:homology}
If $\gamma \in H_1(\bP_1^\circ; \R)$ is non-zero then 
$$
\lim_{n \to +\infty} P \circ \cap_\ast\big(\phi^n(\gamma)\big)=P(\mu^u) 
\quad \text{and} \quad
\lim_{n \to +\infty} P \circ \cap_\ast\big(\phi^{-n}(\gamma)\big)=P(\mu^s).
$$
Moreover, there is a descending sequence of subspaces indexed by $\N$
$$H_1(\bP_1^\circ; \R)=S_0 \supset S_1 \supset S_2 \ldots \quad \text {so that:}$$
\begin{enumerate}
\item Each $S_{j+1}$ is codimension one in $S_j$ (i.e., $S_{j+1}$ is the kernel of a surjective linear map $S_j \to \R$).
\item The intersection $\bigcap_{j\geq 0} S_j$ is the zero subspace $\{\0\}$.
\item For any $\gamma \in S_j \smallsetminus S_{j+1}$ the sequence
$\frac{n^{j+\frac{3}{2}}}{(\lambda_1^u)^n} \cap_\ast\big(\phi^n(\gamma)\big)$
converges in $\sH$ to a non-zero scalar multiple of $\mu^u$.
\end{enumerate}
\end{theorem}
We remark that statement (3) also holds in the opposite direction, i.e., 
$\frac{n^{j+\frac{3}{2}}}{(\lambda_1^s)^{-n}} \cap_\ast\big(\phi^{-n}(\gamma)\big)$
converges to a multiple of $\mu^s$, but the rate given by $j$ is determined by a different sequence of nested subspaces. 
This theorem is a consequence of our Theorem \ref{thm:parabolic intersections} which is stronger in that it gives a formula for
the constant appearing in the limit in statement (3). The proof of Theorem \ref{thm:homology} contains a concrete description of the subspaces $S_j$
and gives a slightly different formula for the limiting constant; see \eqref{eq:Thm 4 constant}.

Again we could consider trying to extend this type of convergence to a larger space, and $\sP \sH$ itself would be a natural candidate but this space also contains fixed points corresponding to measured foliations on $\bP_c$. It is not clear to the author if there is a natural intermediate space between
$P \circ \cap_\ast\big(H_1(\bP_1^\circ; \R)\big)$ and $\sP \sH$ in which we see $P(\mu^u)$ and $P(\mu^s)$ as the global attractor and repeller.

We will now briefly discuss the method we use to prove these results. The main idea is to completely understand algebraic intersection numbers,
and we provide a formula for computing the algebraic intersection numbers between curves on $\bP_1$. It first needs to be observed that the surfaces $\bP_c$ for $c \geq 1$ are canonically homeomorphic, with the homeomorphisms coming from viewing $\bP_c$ for $c \geq 1$ as a parameterized deformation of translation surfaces. Fixing a curve $\gamma \in \bP_1$ representing a homology class we can use these canonical homeomorphisms to obtain corresponding curves in $\bP_c$. We observe that the holonomies of these curves measured on $\bP_c$ denoted $\hol_c \gamma \in \R^2$ depend polynomially in $c$. Thus this quantity makes sense for all $c \in \R$. In Lemma \ref{lem:intersection parabola}, we show that for any two curves $\gamma$ and $\sigma$ representing homology classes on $\bP_1$, their algebraic intersection number is given by 
\begin{equation}
\label{eq:intersection}
\gamma \cap \sigma = \frac{1}{2 \pi} \int_0^\pi \big((\hol_{\cos t}~\gamma) \wedge (\hol_{\cos t}~\sigma) \big) (1-\cos t)~dt.
\end{equation}
This result gives a mechanism to reduce the Theorems above to questions involving the asymptotics of certain trigonometric integrals.

\subsection*{Organization of article}
\begin{itemize}
\item In \S \ref{sect:background} we provide a condensed description of main ideas we use from the theory of translation surfaces. We give more formal descriptions of the homological spaces mentioned above.
\item In \S \ref{sect:subgroup deformation} we investigate the continuous family of representations that arises out of considering the Veech groups of the surfaces $\bP_c$. This family of representations is studied in the abstract and we prove results about eigenvalues that are crucial for our later arguments.
\item Section \ref{sect:parabola surface} addresses the geometry and dynamics of $\bP_1$.
\begin{itemize}
\item In \S \ref{sect:construction} we review the construction of the surfaces $\bP_c$.
\item In \S \ref{sect:homology} we compute generating sets for the homological spaces we work with.
\item In \S \ref{sect:deformation} we explain how holonomies of homology classes deform as we vary $c$ and formalize our intersection number formula described in \eqref{eq:intersection}.
\item In \S \ref{sect:affine} we describe the affine automorphism groups of the surfaces $\bP_c$ for $c \geq 1$.
\item In \S \ref{sect:algebraic} we prove Theorem 
\ref{thm:parabolic intersections}, which is the most important result in the paper:
it gives an asymptotic formula for algebraic intersections of the form $\phi^n(\gamma) \cap \sigma$.
We use it to prove Theorem \ref{thm:homology}.
\item In \S \ref{sect:intersection} we prove our asymptotic formula for intersections of cylinder described by Theorem \ref{thm:mixing}. The key is the observation that
the area of intersection of two cylinders is largely governed by algebraic intersection numbers between the core curves.
\item In \S \ref{sect:geometric} we consider geometric intersection numbers and prove Theorem \ref{thm:geometric}.
\end{itemize}
\end{itemize}

\section{Background and notation}
\label{sect:background}
A translation surface is a topological surface with an atlas of charts to the plane so that the transition functions are translations. Equivalently, an translation surface $\bS$ is a surface whose universal cover is equipped with 
a local homeomorphism called the {\em developing map} $\dev$ from the universal cover $\tilde \bS$ to $\R^2$ so that for any deck transformation $\Delta:\tilde \bS \to \tilde \bS$, there is a translation $T:\R^2 \to \R^2$ so
that $\dev \circ \Delta=T \circ \dev$. Such a surface should be considered equivalent to the surface obtained by 
post-composing the developing map with a translation. Note that this definition does not allow for cone singularities on the surface but they may be treated as punctures.

In this paper we will be considering surfaces of infinite genus, but our surfaces will be decomposable into countably many triangles. Each translation surface we consider $\bS$ will be a countable disjoint union of triangles with edges glued together pairwise by translations in such a way so that each point on the interior of an edge has a neighborhood isometric to an open subset of the plane. (Figure \ref{fig:parabola2} depicts such a decomposition for $\bP_1$.) We will use $\Sigma$ to denote the {\em singularities} which are the equivalence classes of the vertices of the triangles in $\bS$. The surface $\bS \smallsetminus \Sigma$ is a translation surface, while $\bS$ itself is not a surface if countably many triangles meet at a singularity.

Let $R$ be a ring containing $\Z$ such as $\Z$ or $\R$.
For us, first relative homology over $R$ is $H_1(\bS,\Sigma;R)$ may be viewed as the $R$-module generated by oriented edges of the triangulation of $\bS$ and subject to the conditions that the sum of two copies of the same edge with opposite orientations is zero as is the sum of edges oriented as the boundary of a triangle. 
Less formally, $H_1(\bS,\Sigma;R)$ may be viewed as homology classes consisting of finite weighted sums of curves joining points in $\Sigma$ and closed curves subject to the condition that they pass through only finitely many triangles. Any such curve is homotopic to a union of edges and the resulting class in $H_1(\bS,\Sigma;R)$
is independent of the choice of such a homotopy.

Let $\bS^\circ$ denote $\bS \smallsetminus \Sigma$.
We define $H_1(\bS^\circ;\Z)$ to be the abelianization of $\pi_1(\bS^\circ)$.
 Note that by compactness a closed curve in $\bS^\circ$ only intersects finitely many triangles defining $\bS$. We define $H_1(\bS^\circ;R)$ to be $H_1(\bS^\circ;\Z) \otimes_\Z R.$ Again, we may consider $H_1(\bS^\circ;R)$ to represent all finite sums of closed curves in $\bS^\circ$ weighted by elements of $R$.

We use $\cap$ to denote the algebraic intersection number 
$$\cap:H_1(\bS^\circ;\Z) \times H_1(\bS,\Sigma;\Z) \to \Z$$
which is bilinear. We follow the convention that $\alpha \cap \beta$ is positive if 
$\alpha$ is moving rightward and $\beta$ is moving upward. With $R$ as above, algebraic intersection extends to a bilinear map
$$\cap: H_1(\bS^\circ;R)  \times H_1(\bS,\Sigma;R) \to R.$$

Let $\gamma:[0,1] \to \bS$ be a curve and let $\tilde \gamma:[0,1] \to \tilde \bS$ be a lift to the universal cover. The holonomy vector of $\gamma$ is 
$$\hol~\gamma=\dev \circ \tilde \gamma(1) - \dev \circ  \tilde \gamma(0).$$
Observe that this quantity is independent of the choice of lift.
Holonomy yields linear maps 
$$\hol:H_1(\bS,\Sigma;\R) \to \R^2 \quad \text{and} \quad \hol:H_1(\bS^\circ;\R) \to \R^2$$
which send a weighted sum of curves to the corresponding weighted sum of holonomy vectors.
We use $\hol$ for both maps because the following map diagram commutes
\begin{center}
\begin{tikzcd}[column sep=small]
H_1(\bS^\circ;\R) \arrow{rr} \arrow[swap]{dr}{\hol}& &H_1(\bS,\Sigma;\R) \arrow{dl}{\hol}\\
& \R^2 & 
\end{tikzcd}
\end{center}
where $H_1(\bS^\circ;\R) \to H_1(\bS,\Sigma;\R)$ is induced by the inclusion $\bS^\circ$ of into $\bS$.


Let $\bS$ be a translation surface built as above by gluing together triangles in a countable set ${\mathcal T}$. Let $f:\bS \to \bS$ be a homeomorphism preserving the singular set $\Sigma$ and satisfying the condition that for any $T \in {\mathcal T}$, the image $f(T \smallsetminus \Sigma)$ intersects only finitely many triangles in ${\mathcal T}$. Then $f$ acts naturally on the homological spaces defined above.
We denote this action by $f_\ast$.
Let $\tilde f:\tilde \bS \to \tilde \bS$ be a lift to the universal cover. We say that $f$ is an {\em affine automorphism} of $\bS$ if there is an affine map $A:\R^2 \to \R^2$ so that
$$\dev \circ \tilde f = A \circ \dev.$$
This notion does not depend on the lift $\tilde f$. The derivative of an affine automorphism $f:\bS \to \bS$
is the element $Df \in \GL(2, \R)$ which is given by $A$ post-composed by a translation (so that the origin is fixed). 
We say $f$ is {\em hyperbolic} if $Df$ is a hyperbolic matrix, i.e., $Df$ has distinct eigenvalues
and their absolute values differ.
Observe that if $f$ is an affine automorphism then for any homology class $\sigma$ we have
$$\hol~f_\ast(\sigma)=Df \cdot \hol(\sigma).$$
The collection of all affine automorphisms of $\bS$ forms a group which we denote by $\Aff(\bS)$. In this paper, the {\em Veech group} of $\bS$ is $D\, \Aff(\bS) \subset \GL(2,\R)$.
(Some papers consider the Veech group to be $D\, \Aff(\bS) \cap \SL(2,\R)$.)

\section{A subgroup deformation}
\label{sect:subgroup deformation}

For each real $c \in \R$ let $\Gamma_c$ denote the subgroup of $\SL^\pm(2,\R)$ generated by $-I$ together with
the involutions
\begin{equation}
\label{eq:involutions}
A_c=\left[\begin{array}{cc}
-1 & 0 \\
0 & 1
\end{array}\right],
\quad
B_c=
\left[\begin{array}{cc}
-1 & 2 \\
0 & 1
\end{array}\right],
\quad \textrm{and} \quad
C_c=
\left[\begin{array}{cc}
-c & c-1 \\
-c-1 & c
\end{array}\right].
\end{equation}
We will later see that $\Gamma_c$ realizes the Veech group $D \Aff(\bP_c)$ for $c \geq 1$, but in this section we will be interested in proving
some results about how the deformation $\Gamma_c$ effects hyperbolic elements of $\Gamma_1$.

The group $PGL(2,\R)$ is naturally identified with the isometry group of the hyperbolic plane, and each of these three matrices act by reflections in a line in the hyperbolic plane (called the {\em axis} of the reflection).
Figure \ref{fig:deformation} depicts the axes of reflections corresponding to the matrices above. Observe that the relationship between the lines $A$ and $C$ changes as $c$ varies. We will be concerned with
the groups $\Gamma_c$ when $c \geq -1$. Observe that $C_{-1} = -B_{-1}$. When $-1<c<1$, the axes of reflections $A_c$ and $C_c$ intersect at an angle of $\theta$ where $\cos \theta=c$. When $\theta=\frac{2 \pi}{n}$, the group $\Gamma_c$ is discrete with the triangle formed by the reflection axes forming a fundamental domain for the group action.
When $c=1$, these reflection axes become tangent at infinity, and $\Gamma_c$ becomes the congruence two subgroup of $\SL^\pm(2,\Z)$ whose fundamental domain is the ideal triangle enclosed by the three reflection axes. When $c \geq 1$, the group stays discrete but is no longer a lattice since the fundamental domain which consists of the enclosed ultra-ideal triangle has infinite area.

\begin{figure}
\includegraphics[height=2.5in]{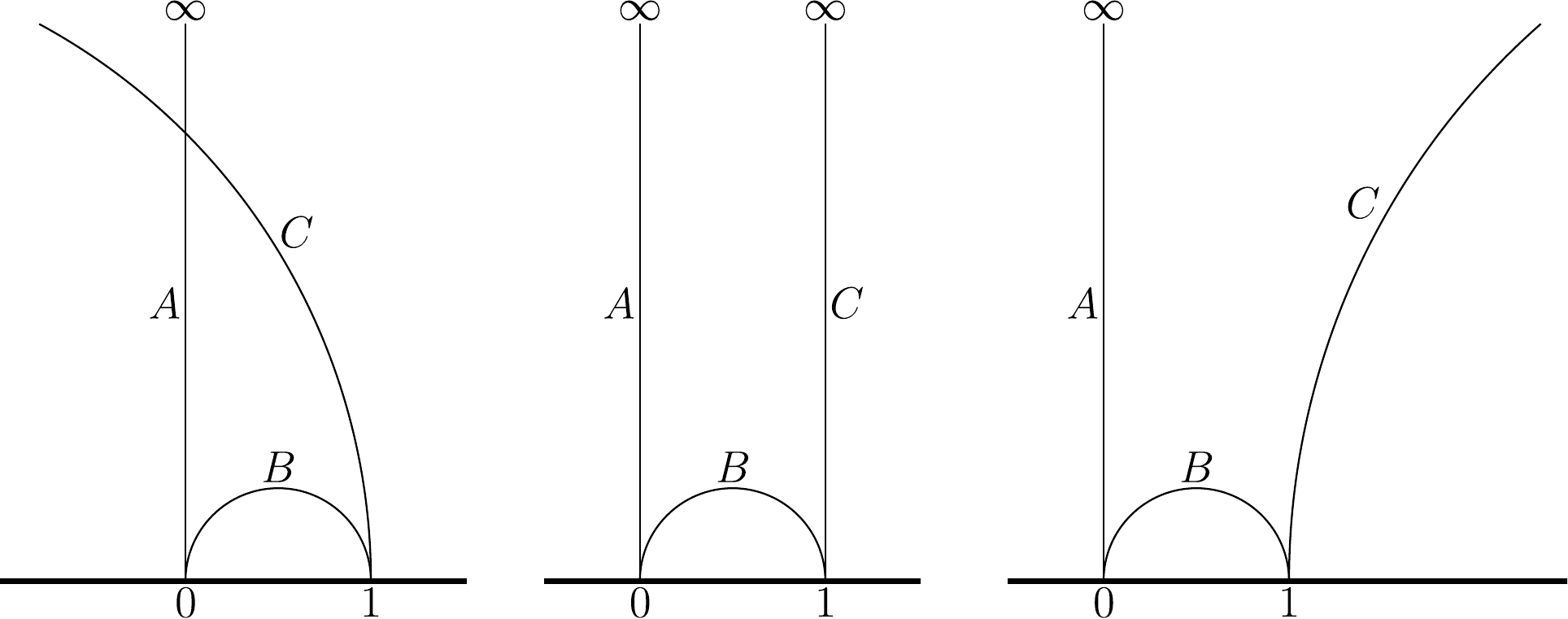}
\caption{The fixed point sets of $A_c$, $B_c$ and $C_c$ in the upper half plane model depicted from left to right for the cases of 
$c = \cos\frac{\pi}{4}$, $c = 1$, and $c = \frac{5}{4}$ from left to right.}
\label{fig:deformation}
\end{figure}

We will think of these groups as a continuous family of representations of the abstract group 
\begin{equation}
\label{eq:sG}{\mathcal G}=(C_2 \ast C_2 \ast C_2) \times C_2,
\end{equation}
where $C_2=\Z/2 \Z$ is the cyclic group of order two. 
We define the representation
\begin{equation}
\label{eq:rho}
\rho:{\mathcal G} \to \SL^\pm(2,\Z[c]) \quad \text{sending the generators of ${\mathcal G}$ to $A_c$, $B_c$, $C_c$ and $-I$ respectively.}
\end{equation}
For any $z \in \C$, we define 
\begin{equation}
\label{eq:rhoc}
\rho_{z}: {\mathcal G} \to \SL^\pm(2, \C); \quad g \mapsto \rho(g)(z),
\end{equation}
i.e., we evaluate $\rho(g)$ with $c=z$. Below and throughout this paper, we abuse notation by identifying the free variable $c$ temporarily with a constant,
so we will simply write $\rho_c$ rather than $\rho_z$. 

The following is the main result we will need about this family of representations.

\begin{lemma}
\label{lem:eigenvalue}
Let $g \in {\mathcal G}$ be such that $\rho_1(g)$ has a real eigenvalue $\lambda^u_1$ with $|\lambda^u_1| > 1$.
For each $c \in \R$, let $\lambda_c^u$ denote the choice of an eigenvalue realizing the maximum value of the absolute value of an eigenvalue
of $\rho_c(g)$. Then:
\begin{enumerate}
\item We have $|\lambda_c^u| < |\lambda_1^u|$ whenever $-1 \leq c <1$.
\item The function $c \mapsto \lambda_c^u$ is differentiable at $c=1$ and the constant
$\beta=\frac{1}{\lambda_1^u} \big[\frac{d}{dc} \lambda_c^u\big]_{c=1}$
is positive.
\end{enumerate}
\end{lemma}
We remark that the constant $\beta$ appeared in Theorems \ref{thm:mixing} and \ref{thm:geometric}.

The remainder of the section is devoted to a proof of the Lemma and the following consequence
which involves the operator norm $\| \cdot \|$ on real $2 \times 2$ matrices 
defined using the Euclidean metric on $\R^2$.

\begin{proposition}
\label{prop:eigenvalue}
Let $g \in \sG$ be as in Lemma \ref{lem:eigenvalue}. 
Let $c_0 \in \R$ be a number so that $-1 \leq c_0 < 1$. Then there is a constant $\xi \in \R$ with $1<\xi<\lambda_1^u$
so that $\|\xi^{-n} \rho_c(g)^n\|$ converges to zero uniformly for $c \in [-1, c_0]$.
\end{proposition}

The group $\PGL(2,\R)$ is the isometry group of the hyperbolic plane. 
For a hyperbolic $G \in \PGL(2,\R)$ the eigenvalue $\lambda^u$ of largest absolute value has the geometric significance:
\begin{equation}
\label{eq:dist}
\inf_{x \in \H^2} \textit{dist}(x, Gx)=2 \log |\lambda^u|.
\end{equation}
Moreover, this infimum is achieved. The collection of points where this infimum is achieved is a geodesic
in $\H^2$ called the {\em axis} of the hyperbolic isometry $G$. The axis has a canonical orientation
determined by the direction $G$ translates the geodesic. 
If $G$ belongs to a discrete group $\Gamma$,
then this axis projects to a curve of length $2 \log |\lambda^u|$ in the quotient $\H^2/\Gamma$.

We will need to extend these ideas to triangular billiard tables which are not quotients of $\H^2$ by 
a discrete group. Let $\Delta$ be any triangle in $\H^2$ (with the case of interest being the triangles shown in Figure \ref{fig:deformation}).
We label the edges by the $\{A,B,C\}$ and we will use $\ell_a$, $\ell_b$, and $\ell_c$ to denote the bi-infinite geodesics in $\H^2$ that contain
the marked edges. Let $\sG' \cong C_2 \ast C_2 \ast C_2$ represent $\sG$ modulo the central $C_2$ in $\sG$.
Given $\Delta$ we get a representation $\rho_\Delta:\sG' \to \PGL(2, \R)$ where we send the generators
$g_A, g_B, g_C \in \sG'$ to the reflections in $\ell_a$, $\ell_b$, and $\ell_c$ respectively.

Each conjugacy class $[g]$ in $\sG'$ has a representative of the form
\begin{equation}
\label{eq:reflection product}
g = g_{L_0} g_{L_1} \ldots g_{L_{n-1}} \quad \text{with each $L_i \in \{A,B,C\}$}
\end{equation}
where $L_i \neq L_{i+1}$ for any $i \in \{0,\ldots,n-1\}$ and addition taken modulo $n$.
Moreover this representative is unique up to cyclic permutations. We define the {\em orbit-class} $\Omega([g])$ to be the collection of all closed rectifiable curves in $\H^2$ that visit the lines
$\ell_{L_0}$, $\ell_{L_1}$, \ldots, $\ell_{L_{n-1}}$  in that order. We define the {\em orbit-length} of $[g]$ to be
$\sL([g])=\inf_{\gamma \in \Omega([g])} \textit{length}(\gamma)$.

We have the following lemma.

\begin{lemma}
\label{lem:inf}
Fix $\Delta$ and notation as above. Let $g \in \sG'$ and let $[g]$ be its conjugacy class. 
Let $\lambda^u$ denote the eigenvalue of $\rho_\Delta(g)$ with largest absolute value.
Then $2 \log |\lambda^u|\leq \sL([g])$.
\end{lemma}
The proof is essentially the ``unfolding'' construction from polygonal billiards:

\begin{proof}
The statement is certainly true unless $|\lambda^u|>1$. Since eigenvalues are a conjugacy invariant, we may assume that
$g$ is given as in \eqref{eq:reflection product}. If the conclusion is false, 
there is a $\gamma \in \Omega(G)$
with length less than $2 \log |\lambda^u|$. We will draw a contradiction by 
construct a new path $\gamma'$ in
$\H^2$ with length equal to that of $\gamma$ such that $\rho_\Delta(g)$ translates the starting point of $\gamma'$
to the ending point of $\gamma'$ violating \eqref{eq:dist}.

We may assume $\gamma$ begins on the geodesic $\ell_{L_0}$ and then travels to $\ell_{L_1}$ and so on. 
For $i \in \{0, \ldots, n-1\}$, let $\gamma_i$
be the portion of $\gamma$ which travels from $\ell_{L_{i-1}}$ to $\ell_{L_{i}}$ with subscripts taken modulo $n$ so that
$\gamma=\gamma_1 \cup \ldots \cup \gamma_{n-1} \cup \gamma_0$.
Define $h_0, \ldots, h_{n} \in \sG'$ inductively so that
$h_0$ is the identity element of $\sG'$ and $h_{i+1}=h(i) g_{L_i} $. This makes $h_n=g$ as given in \eqref{eq:reflection product}.
For $i \in \{0, \ldots, n-1\}$ set $\gamma'_i = \rho_\Delta(h_i)(\gamma_i)$ and set $\gamma'_n=\rho_\Delta(h_n)(\gamma_0)$. We observe
\begin{enumerate}
\item The collection $\gamma'_0 \cup \ldots \cup \gamma'_n$ is a continuous curve.
\item We have $\gamma'_n=\rho_\Delta(h_n)(\gamma'_0)$. 
\end{enumerate}
Statement (2) holds by construction since $\gamma_0=\gamma'_0$ (as $h_0$ is the identity).
To see statement (1) observe that $\gamma_i$ terminates at $\ell_{L_i}$ at the same point at which $\gamma_{i+1}$ starts.
Thus, $\gamma_i \cup \rho_\Delta(g_{L_i})(\gamma_{i+1})$ passes continuously through $\ell_{L_i}$. It follows
that the arcs $\gamma'_i=\rho_\Delta(h_i)(\gamma_i)$ and $\gamma'_{i+1} = \rho_\Delta(h_i g_{L_i})(\gamma_{i+1})$ are joined at some point on $\rho_\Delta(h_i)(\ell_{L_i})$. 
This proves (1).

From (1) we see that a point on $\gamma'_0$ is joined to the corresponding point on $\gamma'_n$ by a path of length equal to the length of $\gamma$,
and from (2) we see that these points differ by an application of the hyperbolic isometry  $\rho_\Delta(h_n)$. Recalling that $h_n=g$, we see this violates \eqref{eq:dist}.
\end{proof}

We will need to introduce one more idea. The {\em Klein model} for the hyperbolic plane consists of 
$\KH^2=\{(x^2,y^2) \in \R^2 ~|~x^2+y^2<1\}$. The boundary 
$\del \KH^2=\{(x^2,y^2) \in \R^2 ~|~x^2+y^2=1\}$. Geodesics in the Klein model
are Euclidean line segments.
Distance between points in the Klein model may be computed
in two ways. Let $P_1$ and $P_2$ be two points in $\KH^2$, and let $\overline{P_1 P_2}$ be the Euclidean
line through them. Let $Q_1$ and $Q_2$ be the two points of $\del \KH^2 \cap \overline{P_1 P_2}$, chosen so that
$P_1$ is closest in the Euclidean metric to $Q_1$. Then the distance between $P_1$ and $P_2$ is given
by the logarithm of the cross ratio,
\begin{equation}
\label{eq:cr}
\textit{dist}_{\KH^2} (P_1,P_2)=\log \Big(\frac{\textit{dist}_{\R^2}(P_1,P_2) \textit{dist}_{\R^2}(Q_1,Q_2)}
{\textit{dist}_{\R^2}(P_1,Q_1) \textit{dist}_{\R^2}(P_2,Q_2)}\Big).
\end{equation}
Alternately, we can compute distance by using the metric tensor $ds$. 
\begin{equation}
\label{eq:metric_tensor}
ds=\sqrt{\frac{dx^2+dy^2}{1-x^2-y^2}+\frac{(x dx+y dy)^2}{(1-x^2-y^2)^2}}.
\end{equation}
The distance between two points can be computed by integrating the metric tensor over the geodesic path
between them. See \cite{CFKP} or any hyperbolic geometry text for more details.

As discussed above the projectivization of $\rho_c(\sG)$ to a subgroup of $\PGL(2,\R)$ is generated by the reflections 
in the sides of a triangle $\Delta_c$ in $\H^2$. See Figure \ref{fig:deformation}.
When $c=\cos \theta \leq 1$, the triangle has two ideal vertices and 
one vertex with angle $\theta$.  
For our purposes, we will think of $\Delta_c \subset \KH^2$. We define
\begin{equation}
\Delta_c=\textit{Convex Hull}(\{P_1,P_2,P_3\}) \subset  \KH^2
\end{equation}
where $P_1=(-1,0)$, $P_2=(1,0)$, and $P_3=(0,\frac{\sqrt{1+c}}{\sqrt{2}})$. 
It is an elementary check that this triangle is isometric to the triangle used to define the reflection group $\rho_c(\sG)$.
The reflection lines of the elements 
$A_c$, $B_c$, and $C_c$ are given
by $\ell_A=\overline{P_3 P_1}$, $\ell_B=\overline{P_1 P_2}$, and $\ell_C = \overline{P_2 P_3}$ respectively.

Fixing $c$ observe that $\rho_c: \sG \to \SL^\pm(2,\R)$ projectivizes to a representation
$\rho_c': \sG' \to \PGL(2,\R)$.
The generators $g_A, g_B, g_C \in \sG'$ are mapped under $\rho_c'$ respectively to 
$A_c$, $B_c$ and $C_c$ of \eqref{eq:involutions} viewed as elements of $\PGL(2,\R)$.
Then $\ell_A$, $\ell_B$, and $\ell_C$ are the fixed point sets of these hyperbolic reflections.

\begin{proof}[Proof of Lemma \ref{lem:eigenvalue}, statement (1)]
Choose $g \in \sG$ so that $\rho_1(g)$ has a real eigenvalue $\lambda_1^u$ with $|\lambda^u_1|>1$.
Let $G_1=\rho_1(g)$ which we think of as a hyperbolic transformation of $\H^2$. In the case of $c=1$, the group
$\rho_1'(\sG')$ is discrete and $\Delta_1$ is a fundamental domain for the action.
Thus, the projection of the axis of $G_1$ to
$\H^2/\rho_1(\sG)$ minimizes length in the {\em orbit-class} $\Omega(g)$.
In particular, this billiard path $\gamma_1$ realizes the infimum discussed in Lemma \ref{lem:inf}.

Let $|\lambda^u_c|$ denote the greatest absolute value of an eigenvalue of $G_c = \rho_c(g)$. 
We will use Lemma \ref{lem:inf} to show that $|\lambda^u_c|<|\lambda_1^u|$ for all $-1 \leq c < 1$. 
In coordinates on the closure of $\KH^2$ consider the linear map
\begin{equation}
\label{eq:mc}
M_c~:~\overline{\KH^2} \to \overline{\KH^2}~:~(x,y)\mapsto\left(x,\frac{y\sqrt{1+c}}{\sqrt{2}}\right).
\end{equation}
Observe from our formula for the vertices that $M_c(\Delta_1)=\Delta_c$. We claim that $M_c$ shortens every line segment in $\Delta_1$ except line segments contained in the side
$\overline{P_1 P_2}$. Consider a segment $\overline{XY}$ with finite length in $\Delta_1$. Let $\overline{PQ}$ be
the geodesic containing $\overline{XY}$, so that $P,Q \in \del \KH^2$ with $P$ the closest to $X$ as
in the left side of Figure \ref{fig:trimap}. Then by \eqref{eq:cr},
$$\begin{array}{rcl}
\textit{dist}_{\KH^2} (X,Y)& =& {\displaystyle \log \Big(\frac{\textit{dist}_{\R^2}(X,Y) \textit{dist}_{\R^2}(P,Q)}
{\textit{dist}_{\R^2}(X,P) \textit{dist}_{\R^2}(Y,Q)}\Big)} \\
& = &{\displaystyle  \log \Big(\frac{\textit{dist}_{\R^2}(M_c(X),M_c(Y)) \textit{dist}_{\R^2}(M_c(P),M_c(Q))}
{\textit{dist}_{\R^2}(M_c(X),M_c(P)) \textit{dist}_{\R^2}(M_c(Y),M_c(Q))}\Big).}
\end{array}
$$
And, let $P',Q' \in \del \KH^2$ be the points where the geodesic $\overline{M_c(P) M_c(Q)}$ intersects 
the boundary. Then,
$$\textit{dist}_{\KH^2} (M_c(X),M_c(Y))=\log \Big(\frac{\textit{dist}_{\R^2}(M_c(X),M_c(Y)) \textit{dist}_{\R^2}(P',Q')}
{\textit{dist}_{\R^2}(M_c(X),P') \textit{dist}_{\R^2}(M_c(Y),Q')}\Big).$$
It is a standard computation that $\textit{dist}_{\KH^2} (X,Y)>\textit{dist}_{\KH^2} (M_c(X),M_c(Y))$
so long as the line segment $\overline{P' Q'}$ strictly contains the line segment $\overline{M_c(P) M_c(Q)}$, i.e.,
as long as $M_c(P) \neq P'$ or $M_c(Q) \neq Q'$. In particular, the only time this inequality 
could be false is when $PQ \subset \{(x,y)|y=0\}$. Our claim is proved.

\begin{figure}[t]
\begin{center}
\includegraphics[width=3.5in]{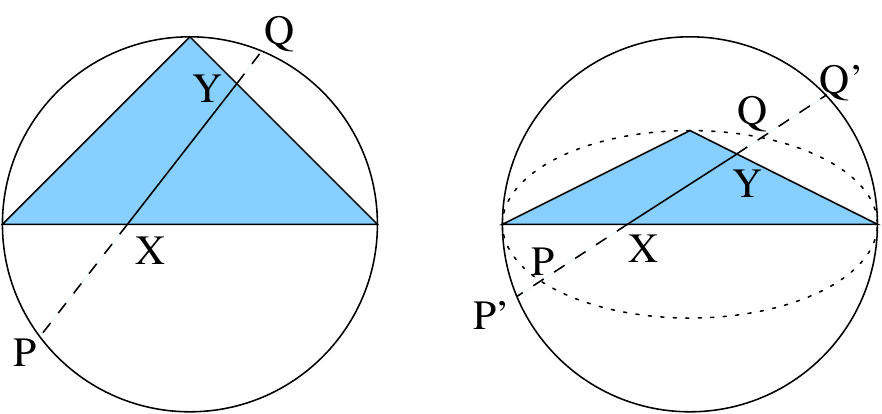}
\caption{A segment and its image under $M_c$. }
\label{fig:trimap}
\end{center}
\end{figure}

To finish the proof, we apply the claim to the billiard path $\gamma_1$ constructed in the first paragraph. No finite length billiard path can have a segment contained in the line $y=0$, therefore 
$$\textit{length}\big( M_c(\gamma_1)\big) < \textit{length}\big( \gamma_1 \big).$$
whenever $-1 \leq c < 1$. Then by Lemma \ref{lem:inf}, for all such $c$, 
$$2 \log |\lambda_c^u| \leq \textit{length}\big( M_c(\gamma_1)\big) < \textit{length}\big( \gamma_1 \big) = 2 \log |\lambda_1^u|.$$
Thus $|\lambda_c^u|<|\lambda_1^u|$. 
\end{proof}

\begin{proof}[Proof of Lemma \ref{lem:eigenvalue}, statement (2)]
Let $g \in \sG$ be chosen so that $\rho_1(g)$ has an eigenvalue $\lambda_1^u$ so that $|\lambda_1^u|>1$.
By possibly replacing $\rho_1(g)$ with $-I \rho_1(g)$ we may assume $\lambda_1^u>1$.
Let $\lambda_c^u$ be the largest eigenvalue of $G_c=\rho_c(g)$, which by continuity of $c \mapsto G_c$
satisfies $\lambda_c^u>1$ in a neighborhood of $1$. We must show
$$\frac{d}{dc} \lambda^u_c\Big|_{c=1}>0.$$
The quantity $\frac{d}{dc} \lambda^u_c$ exists at $c=1$, since entries of $G_c$ vary polynomially in $c$. Thus we can afford
to just look for a one-sided derivative and we will
restrict attention to the case of $c<1$ so that we can make use of the map $M_c: \Delta_1 \to \Delta_c$ of
\eqref{eq:mc}.

Let $m=\frac{\sqrt{1+c}}{\sqrt{2}}$. We have $\frac{d}{dc} m=\frac{1}{2 \sqrt{2+2c}}$,
and $\frac{d}{dc} m|_{c=1}=\frac{1}{4}$.
By \eqref{eq:metric_tensor},
the $\KH^2$ length of the tangent vector ${\bf i}=(1,0)$ at the point $(x,y) \in \KH$ is given
by 
$$I_1=\frac{\sqrt{1-y^2}}{1-x^2-y^2}.$$
The length of the tangent vector $M_c({\bf i})=(1,0)$ at the point $M_c(x,y)=(x,my)$ is given by
$$I_2=\frac{\sqrt{1-m^2 y^2}}{1-x^2-m^2y^2}.$$
We compute 
$$\frac{d}{d c} \frac{I_2}{I_1} \Big|_{c=1}=\frac{y^2(1+x^2-y^2)}{4(1-y^2)(1-x^2-y^2)}$$
which is positive on all of $\KH^2$ other than those points where $y=0$.
Note, we are perturbing $c$ in the negative direction. This says that off the line $y=0$, 
$M_c$ is compressing every horizontal vector enough to be detected
by the first derivative.
Let $J_1$ be the $\KH^2$ length of the tangent vector ${\bf j}=(0,1)$ at the point $(x,y) \in \KH$ and
$J_2$ be the $\KH^2$ length of the vector $M_c({\bf j})=(0,m)$ at the point $M_c(x,y)=(x,my)$.
We have
$$J_1=\frac{\sqrt{1-x^2}}{1-x^2-y^2}
\quad \textrm{and} \quad
J_2=\frac{m\sqrt{1-x^2}}{1-x^2-m^2 y^2}.$$
We compute 
$$\frac{d}{d c} \frac{J_2}{J_1}\Big|_{c=1}=\frac{1-x^2+y^2}{4(1-x^2-y^2)}>0.$$
In this case, $M_c$ is compressing every vertical vector enough to be detected
by the first derivative.

The argument concludes in the same manner as the previous proof. Let $\gamma_1$ be the 
billiard path on $\Delta_1$ corresponding to $G_1$. The argument above tells us that
$\frac{d}{dc} \textit{length}(M_c(\gamma_1))=k>0.$ But for $c<1$, 
$$2 \log \lambda_c^u \leq \textit{length}(M_c(\gamma_1))=\textit{length}(\gamma_1)-k(1-c)+\textit{higher order terms}.$$
By taking a straight forward derivative, we get 
$$2 \log \lambda_c^u=2 \log \lambda_1^u-(1-c)\frac{2 \frac{d}{dc} \lambda_c^u |_{c=1}}{\lambda_1^u}+\textit{higher order terms}.$$
This pair of equations imply
$$\frac{d}{dc} \lambda_c^u |_{c=1} \geq \frac{k \lambda_1^u}{2}>0.$$
\end{proof}

\begin{proof}[Proof of Proposition \ref{prop:eigenvalue}]
Fix a $c_0<1$. The function $c \mapsto |\lambda_c^u|$ is continuous so attains its maximum
in $[-1,c_0]$.  By (1) of Lemma \ref{lem:eigenvalue} we know this maximum is less than $\lambda_1^u$, so we can select $\xi$ so
that $|\lambda_c^u|<\xi<\lambda_1^u$ for all $c \in [-1,c_0]$. We will show that $\xi^{-n} \rho_c(g)^n$ tends to the zero matrix uniformly in the operator norm. 

First observe that it suffices to find an $N' \in \N$ so that $n \geq N'$ implies $\|\rho_c(g)^n\|<\xi^n$
for $c \in [-1,c_0]$.
Suppose this statement is true and fix an $\epsilon>0$. We will find an $N$ so that $n>N$ implies that the operator norm $\|\xi^{-n} \rho_c(g)^n\|<\epsilon$. By continuity of $c \mapsto \|\rho_c(g)^{N'}\|$ we can set
$$
\begin{array}{rcll}
\xi' & = & \sup~\{\|\rho_c(g)^{N'} \|^{\frac{1}{N'}}:~c \in [-1,c_0]\}<\xi & \text{and} \\
M & = & \sup~\big\{\xi^{-n} \|\rho_c(g)^n\|:~\text{$0 \leq n < N'$ and $c \in [-1,c_0]$}\big\}.
\end{array}$$
Then choose $K$ sufficiently large so that 
$$k>K \quad \text{implies} \quad M  \left(\frac{\xi'}{\xi}\right)^{-k N'}  < \epsilon.$$
Any $n>(K+1)N'$ can be written in the form $n=k N'+n'$ for some $k>K$ and some $n'$ satisfying 
$0 \leq n' < N'$ so that
 by submultiplicativity of the operator norm we have
$$\xi^{-n} \|\rho_c(g)^n\| \leq \xi^{-n'} \|\rho_c(g)^{n'}\| \cdot \left(\xi^{-N'} \|\rho_c(g)^{N'}\|\right)^k
\leq M \left(\frac{\xi'}{\xi}\right)^{-k N'} < \epsilon.$$

The main idea is to continuously diagonalize the matrix action of $\rho_c(g)^n$. However there is a minor
difficulty that for some values of $c$ the matrix might not be diagonalizable.
The matrices $\rho_c(g)$
have constant determinant $\pm 1$ and so the eigenvalues are distinct except in the case that the determinant is one and the trace is $\pm 2$. 
There are only finitely many $c$ at which the trace of $\rho_c(g)$ is $\pm 2$ since the trace is polynomial in $c$ and 
is not constant (by statement (2) of Lemma \ref{lem:eigenvalue}).

Let $c_1, \ldots, c_k$ denote the finitely many values of $c$ in $[-1,c_0]$ for which $\rho_{c}(g)$ 
has trace $\pm 2$. Observe that if $P \in \SL(2, \R)$ is a any matrix with determinant one and trace $\pm 2$, then
the operator norm $\|P^n\|$ grows asymptotically linearly. We conclude that for each $i \in \{1, \ldots, k\}$ there is a $N_i$ so that $n>N_i$ implies that $\|\rho_{c_i}(g)^n\| < \xi^{n}$. Then by continuity and submultiplicativity of the operator norm (similar to the argument above), we can find an $r_i>0$ and a $N_i'>N_i$ so that 
\begin{equation}
\label{eq:parabolic bound 1}
\|\rho_{c}(g)^n\| < \xi^n
\quad \text{when $n>N_i'$ and $|c-c_i| \leq r_i$.}
\end{equation}

It remains to consider the complement $[-1,c_0] \smallsetminus \bigcup_{i=1}^k (c_i-r_i, c_i+r_i)$. Let
$J$ be the closure of this complement which is a finite union of closed intervals.
On each $J$ the matrix $\rho_c(g)$ is continuously diagonalizable. That is, we have a continuous
functions $\v^u, \v^s: J \to \C^2$ and continuous $\lambda^u,\lambda^s:J \to \C$ with $\lambda^u_c \neq \lambda^s_c$ for all $c \in J$ so that
$$\rho_c(g)(\v^u_c)=\lambda^u_c \v^u_c \quad \text{and} \quad
\rho_c(g)(\v^s_c)=\lambda^s_c \v^s_c \quad \text{for $c \in J$}.$$
(Note that at this point $u$ and $s$ are just being used to distinguish eigenvectors and do not necessarily correspond to expanding and contracting directions.)
Let $P_c^u$ be the projection matrix so that $\v^u$ is an eigenvector with eigenvalue $1$ and 
so that $\v^s$ is an eigenvector with eigenvalue zero. Let $P_c^s$ be the matrix with the $\v^u$
and $\v^s$ playing opposite roles. Then $P_c^s$ and $P_c^u$ vary continuously for $c \in J$ and we have
\begin{equation}
\label{eq:matrix splitting}
\rho_c(g)^n = (\lambda_c^u)^n P_c^u + (\lambda_c^s)^n P_c^s \quad \text{for all $n \in \N$ and $c \in J$}.
\end{equation}
Letting $M$ denote the supremum of the operator norms of all matrices of the form $P_c^u$ and
$P_c^s$ with $c \in [-1,c_0]$ we see that
$$\|\rho_c(g)^n\| \leq M (|\lambda_c^u|^n + |\lambda_c^s|^n) \leq 2 M (\xi')^n$$
where $\xi'$ is the supremum of $|\lambda_c^u|$ and $|\lambda_c^s|$ over $c \in J$. 
Then by definition of $\xi$ we have $\xi'<\xi$ and thus $2 M (\frac{\xi'}{\xi})^n$ tends exponentially to zero. In particular, there is a $N_J$ so that $n>N$ implies that $\|\rho_c(g)^n\|<\xi^n$ for all $c \in J$. 

Setting $N' = \max \{N_1', \ldots, N_k',N_J\}$ we see that for any $n>N'$ we have $\|\rho_c(g)^n\|<\xi^n$.
From our first observation this proves the proposition.
\end{proof}

\section{The parabola surface}
\label{sect:parabola surface}

\subsection{Construction}
\label{sect:construction}
We follow the construction of the parabola surface as an affine limit of Veech's surfaces built from $2$ regular $n$-gons
as described in \cite{Higl1}.
Consider the regular $n$-gon in $\R^2$ to be the convex hull of the orbit of $(1,0)$ under the rotation matrix 
\begin{equation}
\label{eq:R matrix}
R_t=\left(\begin{array}{rr}
\cos t & -\sin t \\
\sin t & \cos t
\end{array}\right) \quad \text{where $t=\frac{2\pi}{n}$.}
\end{equation}
Three points on this orbit are given by 
$$R_t^{-1}(1,0)=(\cos t, -\sin t), \quad (1,0) \quad \text{and} \quad R_t^{-1}(1,0)=(\cos t, \sin t).$$
There is an affine transformation $C_t:\R^2 \to \R^2$ of the plane which carries these three points to 
$(-1, 1)$, $(0,0)$ and $(1,1)$, respectively and a calculation shows 
\begin{equation}
\label{eq:C1}
C_t(x,y)=\left(\frac{y}{\sin t}, \frac{x-1}{\cos t-1}\right).
\end{equation}
The image of the regular polygon under $C_t$ is the polygon $Q_c^+$ whose vertices lie in the orbit of $(0,0)$ under $T_c =C_t \circ R_t \circ C_t^{-1}$ which we compute
to be the affine map 
\begin{equation}
\label{eq:T}
T_c: \R^2 \to \R^2; \quad \left(\begin{array}{r}
x \\
y
\end{array}\right) \mapsto \left(\begin{array}{rr}
c & c-1 \\
c+1 & c
\end{array}\right) \left(\begin{array}{r}
x \\
y
\end{array}\right)+\left(\begin{array}{r}
1 \\
1
\end{array}\right)
\quad \text{where $c=\cos t$.}
\end{equation}
Observe that because $T_c$ is an affine transformation defined with coefficients in $\Z[c]$, for any $k \in \Z$, the $k$-th vertex $T_c^k(0,0)$ of $Q_c^+$ has coordinates which are polynomial in $c$.

From the above it can be observed that when $c=\cos \frac{2 \pi}{n}$ that $Q_c^+$ is an $n$-gon. When $c=1$, the $T_c$-orbit is given by $T_c^n(0,0)=(n,n^2)$, and we interpret $Q_c^+$ as a polygonal parabola.
Similarly, when $c>1$, $Q_c^+$ should be interpreted as a polygonal hyperbola. 

We let $Q_c^-$ be the image of $Q_c^+$ under rotation by $\pi$ about the origin. To form a surface $\bP_c$ for some $c=\cos \frac{2 \pi}{n}$ or $c \geq 1$,
we identify each edge $e$ of $Q_c^+$ {\em by translation} with the image of the edge of $Q_c^-$ obtained by applying this rotation. These surfaces are depicted in Figure \ref{fig:s1}.

\subsection{Homological generators}
\label{sect:homology}

We use $\Sigma$ to denote the collection of two singularities of $\bP_1$.

\begin{proposition}
\label{prop:relative generators 1}
The saddle connections in the common boundaries of the parabolas $Q_1^+$ and $Q_1^-$ generate $H_1(\bP_1,\Sigma; \Z)$. 
\end{proposition}

We denote these saddle connections by $\sigma_i$ for $i \in \Z$. See Figure \ref{fig:parabola2}.

\begin{figure}
\includegraphics[height=3in]{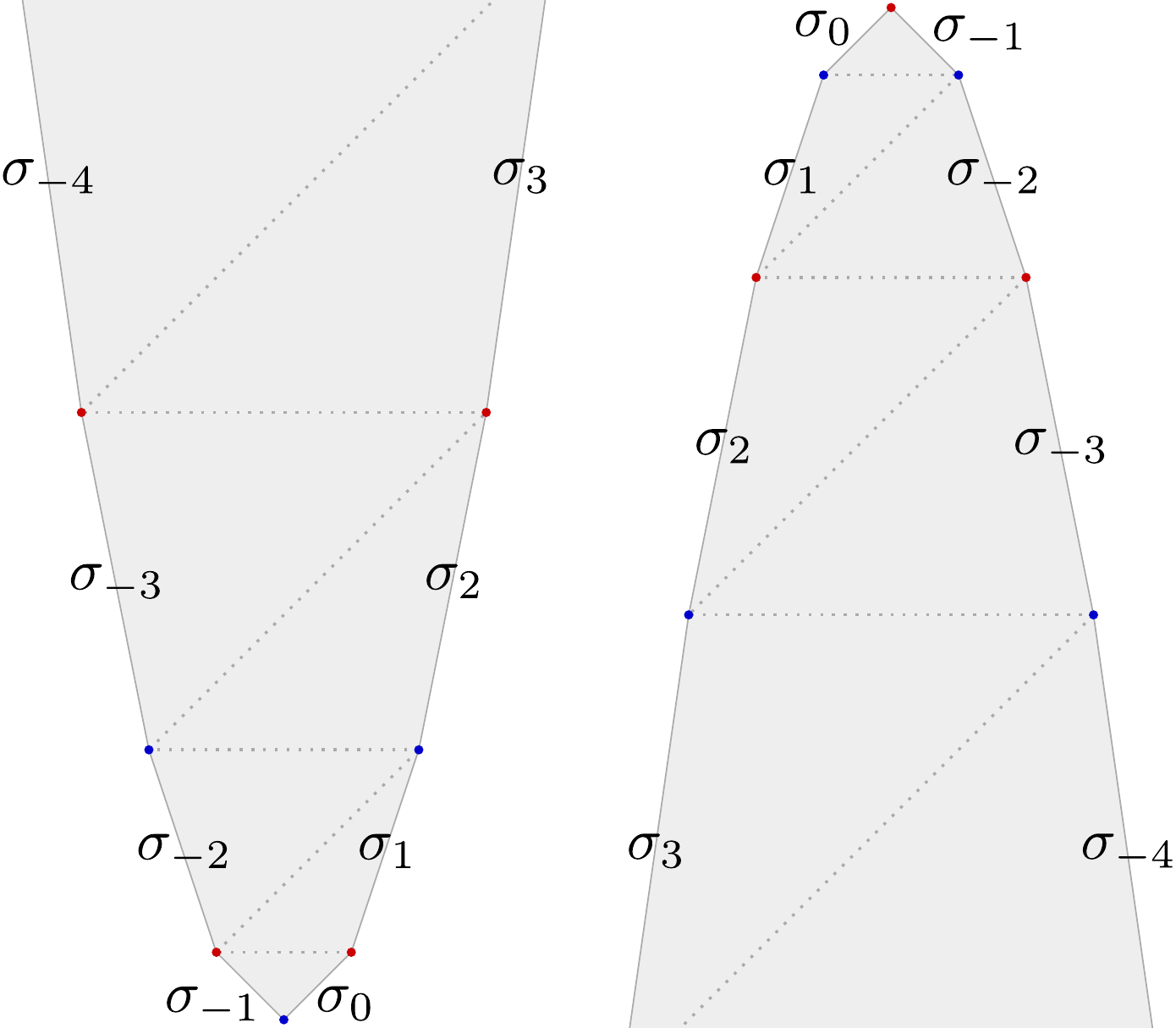}
\hfill
\includegraphics[height=3in]{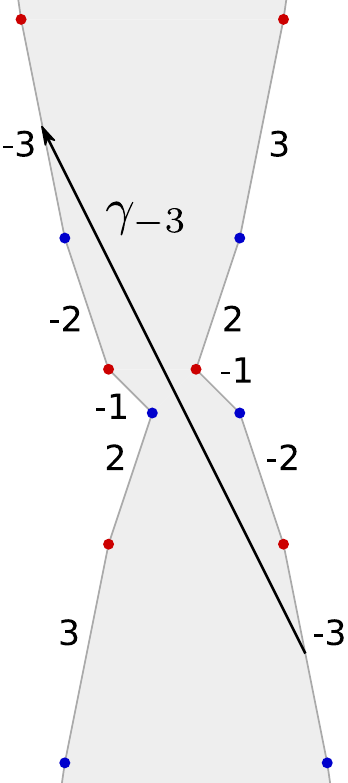}
\caption{Left: the surface $\bP_1$ with saddle connections $\sigma_i$ labeled. Right: the closed geodesic $\gamma_{-3}$.}
\label{fig:parabola2}
\end{figure}

\begin{proof}
The surface $\bP_1$ is triangulated by the saddle connections in the set $\{\sigma_i\}$ together with horizontal and slope one saddle connections. See Figure \ref{fig:parabola2}.
Observe that by using the relation that the sum of edges around a triangle is zero, we can inductively write each horizontal and slope one saddle connection as a sum of the $\sigma_i$. 
\end{proof}

For each integer $j \neq 0$ we define $\gamma_{j}$ to be the closed geodesic which travels within $Q_1^+$ from the midpoint of $\sigma_0$ to the midpoint of $\sigma_{j}$ and then travel back within $Q_1^-$. 
See the right side of Figure \ref{fig:parabola2}.
Let $\bP_1^\circ =\bP_1 \smallsetminus \Sigma$.

\begin{proposition}
\label{prop: gamma generate}
The curves $\gamma_{j}$ generate $H_1(\bP_1^\circ; \Z)$. 
\end{proposition}
\begin{proof}
Consider the polygonal region $R$ in $\R^2$ formed by gluing together $Q_1^+$ and $Q_1^-$ (not including vertices) along the interior of the edge $\sigma_0$. The region $R$ is simply connected. To form $\bP_1^\circ$ we glue along the interiors of edges $\sigma_j$ for $j \neq 0$. Thus the fundamental group of $\bP_1^\circ$ is generated by curves within $R$ which cross over exactly one of these edges. These are our $\gamma_j$ curves, and this means that they also generate the abelianization of the fundamental group $H_1(\bP_1^\circ; \Z)$. 
\end{proof}

\subsection{Deformation and holonomy}
\label{sect:deformation}
The surfaces $\bP_c$ for $c \geq 1$ are all homeomorphic by homeomorphisms $\bP_c \to \bP_{c'}$ respecting the decompositions $\bP_c=Q_c^+ \cup Q_c^-$ and $\bP_{c'}=Q_{c'}^+ \cup Q_{c'}^-$ and sending the orbit $n \mapsto T_c^n(0,0)$ of vertices of $Q_c^+$ to the orbit $n \mapsto T_{c'}^n(0,0)$ of vertices of $Q_{c'}^+$. This uniquely characterizes the homeomorphism up to isotopy.

The paper \cite{Higl1} noted that the surfaces $\bP_c$ for $c \geq 1$ are all naturally homeomorphic, and proved that the surfaces have the same geodesics in a coding sense and have affine automorphism groups which act in the same  way (up to the natural homeomorphism on the surfaces). 

For each $c \geq 1$ we have notion of holonomy on $\bP_c$. Using the canonical homeomorphism $\bP_1 \to \bP_c$
we can evaluate this holonomy on classes in $\bP_1$ giving us a family of holonomy maps
$$\hol_c:H_1(\bP_1,\Sigma,\Z) \to \R^2 \quad \text{defined for $c \geq 1$}.
$$
Observe that $T_c$ has determinant one and entries which are polynomial in $c$. Thus
all vertices of $Q_c^+$ have coordinates which are integer polynomials in $c$. It follows that for any $\sigma \in H_1(\bP_1,\Sigma; \Z)$ the map $c \mapsto \hol_c~\sigma$ lies in $\Z[c]^2$. 
We define the {\em deformation holonomy map} of the family of surfaces $\bP_c$ to be the map
$$\widetilde \hol:H_1(\bP_1,\Sigma;\Z) \to \Z[c]^2$$
so that for $c \geq 1$ we have $\widetilde \hol(\sigma)(c)=\hol_c(\sigma)$.

By tensoring with $\R$ we extend $\widetilde \hol$ to 
\begin{equation}
\label{eq:real holonomy}
\widetilde \hol:H_1(\bP_1,\Sigma;\R) \to \R[c]^2.
\end{equation}
We have the following:
\begin{proposition}
\label{prop:isomorphism 1}
The map $\widetilde \hol$ of \eqref{eq:real holonomy} is an isomorphism of $\R$-modules.
\end{proposition}
\begin{proof}
For $d \geq 0$ let $P_d$ denote the collection of polynomials of degree at most $d$ with coefficients in $\R$. 
Because we have normalized three vertices joined by $\sigma_{-1}$ and $\sigma_0$ in our construction of $Q_c^+$ we can see that
\begin{equation}
\label{eq:hol sigma 0}
\widetilde \hol\, \sigma_0=(1,1) \quad \text{and} \quad \widetilde \hol\, \sigma_{-1}=(1,-1).
\end{equation}
Both these vectors lie in $P_0^2$. 
Recall that $T_c$ carries each vertex of $Q_c^+$ to the subsequent vertex. From the definition of $T_c$ in \eqref{eq:T} we see that for all $j \in \Z$,
$$\widetilde \hol\, \sigma_{j+1}
=
\left(\begin{array}{rr}
c & c-1 \\
c+1 & c
\end{array}\right)
\widetilde \hol\, \sigma_{j}
\quad \text{and} \quad
\widetilde \hol\, \sigma_{j-1}
=
\left(\begin{array}{rr}
c & -c+1 \\
-c-1 & c
\end{array}\right)
\widetilde \hol\, \sigma_{j}.$$
Just considering the $c$s appearing in these matrices we can prove inductively that for $k > 0$,
\begin{equation}
\label{eq:hol sigma k}
\widetilde \hol\, \sigma_k \in \big( (2c)^k, (2c)^k\big)+P_{k-1}^2
\quad \text{and} \quad 
\widetilde \hol\, \sigma_{-k-1} \in \big( (2c)^k, -(2c)^k\big)+P_{k-1}^2.
\end{equation}

To prove the proposition, it suffices to show that for every $n \geq 0$, the restriction of $\widetilde \hol$ to
$\textit{span}_\R~\{\sigma_{-n-1}, \sigma_{-n}, \ldots, \sigma_{n-1}, \sigma_n\}$ is an isomorphism to $P_n^2.$
Noting that there are $2n+2$ vectors in the list which matches the dimension of the space $P_n^2$, we see it suffices
to prove that 
$$\widetilde \hol \big(\textit{span}_\R~\{\sigma_{-n-1}, \sigma_{-n}, \ldots, \sigma_{n-1}, \sigma_n\}\big) \supset P_n^2.$$
We prove this by induction. From \ref{eq:hol sigma 0} we see that the statement is true when $n=0$. Now assuming the statement holds for $n=k-1$,
i.e., $\widetilde \hol\,\big(\textit{span}_\R~\{\sigma_{-k-2}, \ldots, \sigma_{k-1}\}\big) \supset P_{k-1}^2$, we see from \eqref{eq:hol sigma k} that by adding $\sigma_{k}$ and $\sigma_{-k-1}$ to the list, the image of the new span contains $P_{k}^2$.
\end{proof}

Now consider the related map
\begin{equation}
\label{eq:real holonomy 2}
\widetilde \hol:H_1(\bP_1^\circ; \R) \to \R[c]^2.
\end{equation}
To understand it, consider the exact sequence of groups
$$0 \rightarrow H_1(\bP_1^\circ; \R) \xrightarrow{\iota_\ast} H_1(\bP_1,\Sigma;\R) \xrightarrow{\partial_\ast} H_0(\Sigma;\R) \rightarrow H_0(\bP_1^\circ; \R) \rightarrow H_0(\bP_1,\Sigma;\R) \rightarrow 0.$$
This is the standard relative homology long exact sequence in our setting, where we have used the observation that $H_\ast(\bP_1; \R) \cong H_\ast(\bP_1^\circ; \R)$ since the two singularities are isolated and of infinite cone type guaranteeing that $\bP_1$ is homotopy equivalent to $\bP_1^\circ$.
The holonomy map factors through the inclusion $\iota_\ast$, i.e., $\widetilde \hol \circ \iota_\ast = \widetilde \hol$ and the above exact sequence says that
$\iota_\ast$ is an inclusion so we see that the holonomy map \eqref{eq:real holonomy 2} is an isomorphism to a subspace of $\R[c]^2$. To figure out what that subspace is
observe that $\iota_\ast \big(H_1(\bP_1^\circ; \R)\big)=\ker~\partial_\ast$. We can enumerate the two singularities as $s_0$ representing the identified vertices $\{(n,n^2)\}$ of the polygonal parabola with even coordinates and $s_1$ representing the identified vertices with odd coordinates.
The image of $\delta_\ast$ is the collection elements of $H_0(\Sigma;\R)$ of the form $x [s_1]-x[s_0]$ for some $x \in \R$. 
We can recover this value of $x$ as the image of $H_1(\bP_1,\Sigma;\R)$ using the map
$$\epsilon:H_1(\bP_1,\Sigma;\R) \to \R; \quad \sigma \mapsto \frac{1}{2}\big(\partial \sigma(s_1)-\partial \sigma(s_0)\big)$$
which allows us to define the short exact sequence
$$0 \rightarrow H_1(\bP_1^\circ; \R) \xrightarrow{\iota_\ast} H_1(\bP_1,\Sigma;\R) \xrightarrow{\epsilon} \R \rightarrow 0.$$

\begin{proposition}
\label{prop:isomorphism 2}
For $\sigma \in \hol(\bP_1, \Sigma; \R)$ and $(x,y) = \widetilde \hol(\sigma)$ the quantity $y(-1)$ gives
the value of $\epsilon(\sigma)$. Thus we have the isomorphism
$$\widetilde \hol:H_1(\bP_1^\circ; \R) \to \{(x,y) \in \R[c]^2:~y(-1)=0\}.$$

\end{proposition}
\begin{proof}
Since the map sending $\sigma$ to $y(-1)$ is linear, it suffices to check that $\epsilon(\sigma)$ agrees with $y(-1)$ on
our basis $\sigma_j$ for $\hol(\bP_1, \Sigma; \R)$.  Observe that $\epsilon(\sigma_j)=(-1)^j$. 
Using formulas from the proof of Proposition \ref{prop:isomorphism 1} we can compute that
$$\hol_{-1}(\sigma_j)=(-1)^j\, (2j+1,1),$$
and note that the $y$-coordinate matches $\epsilon(\sigma_j)$. 
\end{proof}

We will now prove the integral formula mentioned in the introduction.

\begin{lemma}[Intersection as integration]
\label{lem:intersection parabola}
The algebraic intersection number of any $\gamma \in H_1(\bP_1^\circ; \R)$ and any $\sigma \in H_1(\bP_1,\Sigma; \R)$ is given by
$$\frac{1}{2\pi} \int_0^\pi \big((\widetilde \hol~\gamma) \wedge (\widetilde \hol~\sigma) \big) (1-\cos t)~dt,$$
where the wedge product in the integral yields an element of $\Z[c]$ interpreted as a function of $t$ by setting $c=\cos t$. 
\end{lemma}
\begin{proof}
Observe that the integral expression is bilinear in $\gamma$ and $\sigma$. Algebraic intersection number is bilinear as well, so it suffices to check
the formula on our generating sets for $H_1(\bP_1^\circ; \Z)$ and $H_1(\bP_1, \Sigma; \Z)$.
Observe that $\gamma_j$ (oriented to move from $\sigma_1$ to $\sigma_j$ in $Q_1^+$) intersects $\sigma_j$ with positive sign and $\sigma_0$ with negative sign. 
Thus $\gamma_j \cap \sigma_k = \delta_{j,k} - \delta_{0,k}$ where $\delta_{a,b}$ equals $1$ if $a=b$ is zero otherwise.
We will show that the integral evaluates to the same expression.

First we need to compute $\widetilde \hol~\gamma_j$ and $\widetilde \hol~\sigma_k$. Let $v_n=T_c^n(0,0) \in \Z[c]^2$ be the $n$-th vertex of $Q_c^+$ viewed as a polynomial in $c$.
Observe
$$
\widetilde \hol~\gamma_j = v_{j+1}+v_{j}-v_1-v_0 \quad \text{and} \quad 
\widetilde \hol~\sigma_k = v_{k+1}-v_k.
$$
Using the fact that $T_c =C_t \circ R_t \circ C_t^{-1}$ and the definition of $C_t$ in \eqref{eq:C1}, we see 
$v_n=C_t(\cos nt, \sin nt).$
Therefore
\begin{equation}
\label{eq:c image}
\begin{array}{ll}
\widetilde \hol~\gamma_j = C_t \Big(\cos\big((j+1)t\big)+\cos(jt)-\cos(t)-1, \sin\big((j+1)t\big)+\sin(jt)-\sin(t)\Big),\\
\widetilde \hol~\sigma_k = C_t \Big(\cos\big((k+1)t\big)-\cos(kt), \sin\big((k+1)t\big)-\sin(kt)\Big).
\end{array}
\end{equation}
The transformation $C_t$ is affine and scales signed area by a multiplicative constant. Namely,
$$C_t(\bv) \wedge C_t(\bw)=\frac{1}{{\left(1-\cos t\right)} \sin t} (\bv \wedge \bw)
\quad \text{for any $\bv,\bw \in \R^2$}.$$
Letting $\bv=C_t^{-1}(\widetilde \hol~\gamma_j)$ and $\bw=C_t^{-1}(\widetilde \hol~\sigma_k)$ be the quantities in parenthesis enclosed in \eqref{eq:c image},
we see that the quantity being integrated is
$$\frac{\bv \wedge \bw}{\sin t}=2 \cos\big((k-j) t\big)-2 \cos(k t),$$
where we have done significant simplifying using trigonometric identities. It follows that
$$\frac{1}{2\pi} \int_0^{\pi} \frac{\bv \wedge \bw}{\sin t}~dt=\delta_{j,k}-\delta_{0,k}.$$
as desired.
\end{proof}

For our later discussion of geometric intersection numbers, it will be useful for us to extend our notion
of algebraic intersection number to a bilinear map 
\begin{equation}
\label{eq:extended intersection number}
\cap : H_1(\bP_1,\Sigma; \R) \times H_1(\bP_1,\Sigma; \R); \quad \sigma_1 \cap \sigma_2 =
\frac{1}{2\pi} \int_0^\pi \big((\widetilde \hol~\sigma_1) \wedge (\widetilde \hol~\sigma_2) \big) (1-\cos t)~dt.
\end{equation}
This extension has geometric meaning for saddle connections.

\begin{lemma}
\label{lem:intersection parabola 2}
Let $\sigma_1$ and $\sigma_2$ be saddle connections and let $i(\sigma_1, \sigma_2)$ denote the number
of (unsigned) intersections of $\sigma_1$ and $\sigma_2$ not counting those that occur at the singularities. 
Then
$$\Big|i(\sigma_1,\sigma_2) - |\sigma_1 \cap \sigma_2|\Big| \leq 1.$$
\end{lemma}
\begin{proof}
There two cases to consider. First suppose one of the curves joins a singularity to itself, say $\sigma_1$ has two endpoints at the singularity $s_\ast$.
Then $\sigma_1$ is really a loop and we can apply a homotopy 
only deforming $\sigma_1$ in a small neighborhood of the $s_\ast$ which makes $\sigma_1$ into a new closed loop $\hat \sigma_1$. We will detail a way to obtain $\hat \sigma_1$ so that it is fairly easy to see what is happening. For $r \in (0,\sqrt{2})$
observe that the ball of radius $r$ about $s_\ast$ contains no complete saddle connections.
By choosing $r \in (0,\sqrt{2})$ sufficiently small we can arrange that the ball $B_r(s_\ast)$ 
only intersects $\sigma_1$ and $\sigma_2$ in segments of length $r$ at the start and end of the saddle connections. (The ball will only intersect $\sigma_2$ if it starts or ends at the same singularity.) Chop off the segments of $\sigma_1$ that are within the ball. Since this singularity is an infinite cone singularity, the boundary $\partial B_r(s_\ast)$ is homeomorphic to the real line, so we can join the points of where $\sigma_1$ hits the ball by a unique arc $\partial B_r(s_\ast)$. Call the resulting loop $\hat \sigma_1$.
The added arc may cross $\sigma_2$ if the saddle connection $\sigma_2$ starts or ends at $s_\ast$.
If $\sigma_2$  has one endpoint at $s_\ast$, we have introduced at most one new crossing so that the result holds in this case.
If $\sigma_2$ both starts and ends at $s_\ast$, then the added arc may cross twice, but if it does then the signs ascribed to the intersections are opposite. Again we have shown that $|\sigma_1 \cap \sigma_2|$ is within one of $i(\sigma_1, \sigma_2)$. 

If we can not arrange to be in the first case, then both $\sigma_1$ and $\sigma_2$ join distinct singularities. Orientation is irrelevant for the statement we are trying to prove, so we can assume $\sigma_1$ moves from singularity $s_0$ to singularity $s_1$ and $\sigma_2$ moves from $s_1$ to $s_0$. Let $\gamma_1$ be the curve formed by concatenating $\sigma_1$ and $\sigma_2$. Because our extended definition of $\cap$ is bilinear and alternating, we have
$$\sigma_1 \cap \sigma_2 = \gamma_1 \cap \sigma_2.$$
We can again make $\gamma_1$ into a closed curve  $\hat \gamma_1$ following the method of the previous paragraph. This time we choose $r$ small enough so that the two balls
$B_r(s_0)$ and $B_r(s_1)$ do not intersect and
so that the balls only intersect the saddle connections in initial and terminal segments.
So $\hat \gamma_1$ follows $\sigma_1$ outside of the two balls then wraps around the boundary of $B_r(s_1)$, then
follows a $\sigma_2$ until it hits $B_r(s_1)$ and closes up following the boundary of $B_r(s_0)$. Actually, it is better not to follow $\sigma_2$ exactly,
instead we follow a parallel arc which stays on one side of $\sigma_2$. Again the arcs in the boundary of the balls may have introduced one or two new crossings,
but if we introduce two then they occur with opposite signs. Again we have shown that $|\sigma_1 \cap \sigma_2|$ is within one of $i(\sigma_1, \sigma_2)$. 
\end{proof}

\subsection{Affine automorphisms}
\label{sect:affine}
The affine automorphism group was investigated carefully in \cite{Higl1}.

\begin{theorem}[Theorem 3 \cite{Higl1}]
For $c \geq 1$, the group $D\big(\Aff(\bP_c)\big)$ is generated by $A_c$, $B_c$, $C_c$ and $-I$ as defined in 
\eqref{eq:involutions}. 
\end{theorem}

Since $\bP_c$ never has translation automorphisms, we have the following:

\begin{proposition}[Proposition 5 \cite{Higl1}]
\label{prop:D parabolic}
For $c \geq 1$, the homomorphism $D:\Aff(\bP_c) \to \GL(2,\R)$ is one-to-one.
\end{proposition}

As a consequence we see that for each $c \geq 1$, there is a canonical isomorphism
$$\Phi_c: \sG \to \Aff(\bP_c) \quad \text{so that $D \circ \Phi_c(g)=\rho_c(g)$ for all $g \in \sG$,}$$
where $\sG$ and $\rho_c$ are defined as in \eqref{eq:sG} and \eqref{eq:rhoc} respectively.
The actions are essentially the same in a topological sense:

\begin{theorem}[Theorem 7 \cite{Higl1}]
For $c,c'\geq 1$ and any $g \in \sG$, the automorphisms $\Phi_c(g):\bP_c \to \bP_c$ and $\Phi_{c'}(g):\bP_{c'} \to \bP_{c'}$ are the same up to conjugation by the canonical homomorphism $\bP_c \to \bP_{c'}$
and isotopy.
\end{theorem}

This has the following consequence for our deformation holonomy map.

\begin{proposition}
\label{prop:affine action parabolic} 
Let $\phi=\Phi_1(g) \in \Aff(\bP_1)$. 
Then the induced actions of $\phi$ on $H_1(\bP_1^\circ; \Z)$ and $H_1(\bP_1, \Sigma; \Z)$ satisfy
$$\widetilde \hol \circ \phi(\gamma) = \rho(g)\, \widetilde \hol(\gamma)$$
where $\rho:\sG \to \SL^\pm(2,\Z[c])$ was defined in \eqref{eq:rho}.
\end{proposition}
\begin{proof}
Fixing a $c \geq 1$, we see that $\hol_c \circ \Phi_c(g)=\rho_c(g) \cdot \hol_c(\gamma)$ since $D \circ \Phi_c(g)=\rho_c(g)$. This expression represents
the equation in the proposition evaluated at a specific $c \geq 1$, but we have verified it for uncountably many values (all $c \geq 1$). Fixing any $\gamma$, the expression claims equality of two elements of $\Z[c]^2$. The entries are polynomial and we have verified the equation on infinitely many values, so the equation holds for all $c$.
\end{proof}

\subsection{Asymptotics of algebraic intersections}
\label{sect:algebraic}

In this subsection we state our main result involving asymptotic algebraic intersection numbers of
homology classes. The main ideas involve using $\widetilde \hol$ to convert the homology classes
to elements of $\R[c]^2$, use Proposition \ref{prop:affine action parabolic}  to convert an affine automorphism's action to the action of an element $\rho(g) \in \SL^\pm(2,\Z[c])$, and to use our integral formula to evaluate the intersection numbers.

It is useful to notice that since $\rho_1(g)$ is hyperbolic, the matrices $\rho_c(g)$ are diagonalizable 
in a neighborhood of $c=1$. We can express this diagonalization in the form 
\begin{equation}
\label{eq:local diagonalization}
\rho_c(g)=\lambda^u_c P^u_c+\lambda^s_c P^s_c
\end{equation}
where $\lambda^u_c$ and $\lambda^s_c$ are analytic real valued functions of $c$ corresponding
to the expanding and contracting eigenvalues of $\rho_c(g)$ and $P^u_c$ denote the projection matrices (matrices with one eigenvalue equal one and one equal zero)
with the same eigenvectors as $\rho_c(g)$. Again we just interpret these quantities as defined and analytic in a neighborhood of $c=1$. 

Let $\v,\bw \in \R[c]^2$. We define
\begin{equation}
\label{eq:barwedge}
\v \barwedge \bw = (P^u_c \v) \wedge \bw
\end{equation}
which is a real valued analytic function defined in a neighborhood of $c=1$.  This quantity relates
to constants in our asymptotics and it will be important to know that this function can not be 
identically zero.

\begin{lemma}
\label{lem: barwedge}
Let $\rho_1(g)$ be hyperbolic and define $P^u_c$ and $\barwedge$ as above. Then so long as $\v,\bw \in \R[c]^2$ are both non-zero, the real analytic function $\v \barwedge \bw$ is not identically zero.
\end{lemma}
\begin{proof}
First observe that $P^u_c+P^s_c=I$ so that 
$$\v \barwedge \bw = (P^u_c \v) \wedge (P^u_c \bw+P^s_c \bw)= (P^u_c \v) \wedge (P^s_c \bw).$$
Since the functions are analytic and the eigendirections are transverse, the only way $\v \barwedge \bw$ could be identically zero is if either $P^u_c \v$ was identically zero or $P^s_c \bw$ was identically zero.

We will approach this by contradiction and without loss of generality we may assume $P^u_c \v \equiv 0$.
This says that $\v(c)$ lies in the stable eigenspace of $\rho_c(g)$ for all $c$ sufficiently close to one.
That is,
$$\rho_c(g) \v=\lambda_c^s \v.$$
The entries of both $\rho_c(g) \v$ and $\v$ lie in $\R[c]$. Thus it follows that $\lambda_c^s$ is a rational function, i.e., there are polynomials $p,q \in \R[c]$ so that $\lambda_c^s=\frac{p}{q}$ in a neighborhood of $c=1$ wherever $q(c)\neq 0$ (which holds an open set since $\v$ has non-zero polynomial entries).
Furthermore, we can assume that $\frac{p}{q}$ is reduced in the sense that
$p$ and $q$ 
share no roots in $\C$. 
Since $\det \rho_c(g) \equiv \pm 1$, we know that $\lambda_c^u=\pm \frac{q}{p}$. The sum of the eigenvalues is the trace of $\rho_c(g)$ which we will denote by $t \in \R[c]$. Thus we have the identity
\begin{equation}
\label{eq:sum of squares}
p^2 \pm q^2 = t p q.
\end{equation}
By Lemma \ref{lem:eigenvalue}, we know $\frac{d}{dc} \lambda^u_c|_{c=1} \neq 0$, so that at least one of $p$ and $q$ is non-constant. This means that one of those, say $p$, two has a root $z \in \C$. Then $z$ is also a root of $tpq$. Thus it follows from \eqref{eq:sum of squares} that $z$ must be a root of $q^2$ and therefore also of $q$, but this contradicts our assumption that $p$ and $q$ do not share a common root.
\end{proof}

The following is our main technical result. It describes the asymptotics of homology classes under a hyperbolic affine automorphism.

\begin{theorem}
\label{thm:parabolic intersections}
Suppose $\phi:\bP_1 \to \bP_1$ is hyperbolic affine automorphism with derivative $D \phi=\rho_1(g)$.
Let $\gamma \in H_1(\bP_1^\circ; \R)$ and $\sigma \in H_1(\bP_1, \Sigma; \R)$ be non-zero classes. Define quantities as above in \eqref{eq:local diagonalization} and \eqref{eq:barwedge}.
Let $k \geq 0$ be an integer and $\kappa \in \R$ be nonzero so that the Taylor expansion of $\widetilde \hol \gamma \barwedge \widetilde \hol \sigma$ about $c=1$ is of the form
$$\widetilde \hol \gamma \barwedge \widetilde \hol \sigma = \kappa (c-1)^k+O\big((c-1)^{k+1}\big).$$
(Note that this quantity is not identically zero by Lemma \ref{lem: barwedge}.)
Then the sequence of algebraic intersection numbers $\phi^n_\ast(\gamma) \cap \sigma$ is asymptotic to a constant times $n^{-k-\frac{3}{2}} (\lambda^u_1)^n$, and in fact:
$$\lim_{n \to \infty} \frac{n^{k+\frac{3}{2}}}{(\lambda_1^u)^n} \big(\phi^n_\ast(\gamma) \cap \sigma\big) =
\frac{(-1)^k \Gamma(k+\frac{3}{2}) \kappa \sqrt{2}}{4 \pi \beta^{k+\frac{3}{2}}}$$
where $\Gamma$ denotes the gamma function and $\beta=\frac{1}{\lambda_1^u} [{\textstyle \frac{d}{dc}} \lambda_c^u]_{c=1}$ 
as in Theorem \ref{thm:mixing} of the introduction.
This result also holds for $\gamma, \sigma \in H_1(\bP_1, \Sigma; \R)$ with algebraic intersection numbers computed as in \eqref{eq:extended intersection number}.
\end{theorem}
\begin{proof}
Fix $\gamma$ and $\sigma$. We will compute intersections using the integral in Lemma \ref{lem:eigenvalue} with this being the definition in
the case $\gamma, \sigma \in H_1(\bP_1, \Sigma; \R)$.
Let $c=\cos t$ throughout this proof.
Let $\bv_c = \widetilde \hol~\gamma$ and $\bw_c = \widetilde \hol~\sigma$ which are both elements of $\R[c]^2$. Determine $k$ and $\kappa$ as stated in the theorem.
By our integral formula for intersection numbers and Proposition \ref{prop:affine action parabolic},
$$\phi_\ast^n(\gamma) \cap \sigma = \frac{1}{2\pi} \int_0^\pi \big( (\rho_c(g)^n \bv_c) \wedge \bw_c \big) (1-\cos t)~dt.$$
Here, the quantity $(\rho_c(g)^n \bv_c) \wedge \bw_c$ lies in $\R[c]$ and we integrate with respect to $t$ while taking $c=\cos t$. 

We will be demonstrating that 
\begin{equation}
\label{eq:parabolic goal 1}
\frac{n^{k+\frac{3}{2}}}{2\pi (\lambda^u_1)^{n}} \int_0^\pi \Big(\big( \rho_c(g)^n \bv_c\big) \wedge \bw_c \Big) (1-\cos t)~dt
\end{equation}
is asymptotic to the quantity on the right side of the equation in the theorem.

Let $\lambda^u_c$ be an eigenvalue for $\rho_c(g)$ realizing the maximum absolute value of an eigenvalue. 
From Lemma \ref{lem:eigenvalue} 
we know that $\frac{d}{dc} |\lambda^u_c|\big|_{c=1}>0$ and thus we can find an interval $[0,\epsilon]$ on which
$t \mapsto |\lambda^u_{\cos t}|$ is decreasing and takes values larger than one. We can split the integral of 
\eqref{eq:parabolic goal 1} into two pieces at $\epsilon$. The contribution of the interval $[\epsilon,\pi]$ to 
\eqref{eq:parabolic goal 1} can be written as
\begin{equation}
\label{eq:parabolic goal 2}
\frac{n^{k+\frac{3}{2}}}{2\pi \left(\frac{\lambda^u_1}{\xi}\right)^n} 
\int_\epsilon^\pi \Big( \xi^{-n} \rho_c(g)^n \bv_c\big) \wedge \bw_c \Big) (1-\cos t)~dt,
\end{equation}
where $\xi$ is some quantity so that $1<\xi<|\lambda^u_1|$ so that $\xi^{-n} \|\rho_c(g)^n\|$ tends to zero uniformly for $t \in [\epsilon,\pi]$
as obtainable from Proposition \ref{prop:eigenvalue}.
Observe that the fraction in front of \eqref{eq:parabolic goal 2} decays to zero because of the exponential growth of $\left(\frac{\lambda^u_1}{\xi}\right)^n$.
The integral in \eqref{eq:parabolic goal 2} also tends to zero because $\bv_c$, $\bw_c$ and $1-\cos t$ are continuous in $t$ and therefore bounded in absolute value by a constant. On the other hand, from our use of Proposition \ref{prop:eigenvalue}, the quantity $\xi^{-n} \rho_c(g)^n$ decays to zero uniformly.

Now consider the interval $[0,\epsilon]$. As in \eqref{eq:local diagonalization}, on this interval we can write
$$\rho_c(g)^n = (\lambda_c^u)^n P_c^u + (\lambda_c^s)^n P_c^s \quad \text{for $n \geq 0$.}$$
Thus we can split the contribution of $[0,\epsilon]$ to \eqref{eq:parabolic goal 1} into unstable and stable parts. The stable part has
the form
$$
\frac{n^{k+\frac{3}{2}}}{2\pi (\lambda^u_1)^{n}}
\int_0^\epsilon \Big(\big( (\lambda_c^s)^n P_c^s  \bv_c\big) \wedge \bw_c \Big) (1-\cos t)~dt,
$$
which tends to zero exponentially since $|\lambda_c^s|=|\lambda_c^u|^{-1}<1$ so that the integral is uniformly bounded
while the fraction in front decays exponentially. 

The unstable part is more interesting and using the fact that $\lambda_c^u$ has constant sign for $t \in [0,\epsilon]$
we can write it as
\begin{equation}
\label{eq:unstable part}
\frac{n^{k+\frac{3}{2}}}{2\pi |\lambda^u_1|^{n}}
\int_0^\epsilon |\lambda_c^u|^n
(\bv_c \barwedge \bw_c ) (1-\cos t)~dt
\end{equation}
where $\barwedge$ is used as in \eqref{eq:barwedge}.
We will use a theorem of Erd\'elyi following \cite[Ch. 3 \S 8]{Olver14} to study the asymptotics of the sequence of integrals (temporarily ignoring the fraction in front). The integral can be written in the form 
$$\int_0^\epsilon e^{-n p(t)} q(t)~dt \quad \text{where} \quad p(t)=-\ln |\lambda^u_c| \quad \text{and} \quad q(t)=(\bv_c\barwedge \bw_c) (1-\cos t).$$
To apply Erd\'elyi's theorem we need to know the first few terms of a series expansion for $p(t)$ and $q(t)$ about $t=0$. By the chain rule,
in a neighborhood of $t=0$ we have 
$$p'(t)=\frac{-1}{|\lambda^u_{c}|} \left(\frac{d}{dc} |\lambda_c^u|\right)(-\sin t).$$
Thus we have $p(0)=-\ln |\lambda_1^u|$, $p'(0)=0$, and  
$p''(0)=\frac{1}{|\lambda_1^u|} \cdot [\frac{d}{dc} |\lambda_c^u|]_{c=1}=\beta,$
which is positive by statement (2) of Lemma \ref{lem:eigenvalue}.
In addition, we know that $t=0$ is the location of the minimum of $p(t)$ on $[0,\epsilon]$. 
Recalling that a Taylor series expansion for $\bv_c\barwedge \bw_c= \widetilde \hol~\gamma \barwedge \widetilde \hol~\sigma$ in terms of $c$ was given as a hypothesis of the Theorem, 
we have
$$q(t)=\frac{(-1)^k \kappa}{2^{k+1}} t^{2k+2} + O(t^{2k+4}).$$
By \cite[Theorem 8.1]{Olver14}, the sequence of values of integrals has the asymptotic form
$$e^{-n p(0)}\left[\Gamma\Big(k+\frac{3}{2}\Big) \cdot n^{-k-\frac{3}{2}} \cdot \frac{(-1)^k \kappa}{2^{k+1}} \cdot \frac{1}{2} \cdot \Big(\frac{p''(0)}{2}\Big)^{-k-\frac{3}{2}}+O(n^{-k-2})\right].$$
Plugging these values in for the integral in \eqref{eq:unstable part}, we obtain the limit stated in the theorem.
\end{proof}

\begin{proof}[Proof of Theorem \ref{thm:homology}]
It suffices to prove the second statement, since the statement on convergence of projective classes to $P(\mu^u)$ and $P(\mu^s)$ follows from statement (3) in the theorem.

Fix $\phi$ and notation as above. Given a non-zero $\gamma \in H_1(S,\Sigma; \R)$ observe that 
$P_c^u\, \widetilde \hol \gamma$
has coordinates which are real analytic functions of $c$. As a consequence of Lemma \ref{lem: barwedge}, 
$P_c^u\, \widetilde \hol \gamma$ is not identically zero. Then by analyticity, the following constant is defined 
\begin{equation}
\label{eq:j}
k = \min \left\{k \in \Z: \quad k \geq 0 \quad \text{and}\quad \big[\frac{d^k}{d c^k} P_c^u \widetilde \hol \gamma\big]_{c=1} \neq \0\right\}.
\end{equation}

The value of $k$ can be interpreted as a function $k:H_1(S,\Sigma; \R) \to \Z \cup \{+\infty\}$ where we assign $k(\gamma)$ as
in \eqref{eq:j} except if $\gamma=\0$ in which case we assign $+\infty$. Then we can define the subsets in the theorem to be 
$S_j=\{\gamma:~k(\gamma) \geq j\}$. Fixing $k$ we observe that the map $D_k: \gamma \mapsto \frac{d^k}{d c^k} P_c^u \widetilde \hol \gamma$
is linear and consequently the $S_j$ are subspaces and each $S_{j+1}$ is codimension at most one in $S_j$. 
Recalling Proposition \ref{prop:isomorphism 2}, we see that we can construct pairs of polynomials $(p,q)$ in the image of $\widetilde \hol\big(H_1(S,\Sigma; \R)\big)$
so that the minimal $k \geq 0$ such that $[\frac{d^k}{d c^k} P_c^u (p,q)]_{c=1} \neq \0$
is arbitrary: Such $\big(p(c),q(c)\big)$ need to approximate the stable direction of $\rho_c(g)$ at $c=1$
to order $k-1$ but not order $k$. It follows that $S_{j+1}$ is always codimension one in $S_j$. Finally because of the previous paragraph we know that if $\gamma \neq \0$
then $k(\gamma) \neq 0$ so that $\bigcap_{j \geq 0} S_j=\{\0\}.$ We have proved statements (1) and (2) of the theorem.

To see (3) fix a non-zero $\gamma$ and let $k$ be as in \eqref{eq:j}. We need to show that there is a constant $L$ so that for any $\sigma \in H_1(S,\Sigma; \R)$,
$$\lim_{n \to +\infty} \frac{n^{k+\frac{3}{2}}}{(\lambda^u_1)^n} \big(\phi^n(\gamma) \cap \sigma\big)=L \mu^u(\sigma),$$
and by definition $\mu^u(\sigma)=\u^u_1 \wedge \hol_1(\sigma)$ where $\u^u$ denotes a choice of a unit unstable eigenvector of $\rho_1(g)$.
It suffices to prove this for elements of the form $\sigma_l$ which generate $H_1(S; \Sigma; \R)$ by Proposition \ref{prop: gamma generate}.
Theorem \ref{thm:parabolic intersections} tells us that
\begin{equation}
\label{eq:lim}
\lim_{n \to +\infty} \frac{n^{k+\frac{3}{2}}}{(\lambda^u_1)^n} \big(\phi^n(\gamma) \cap \sigma\big) = \frac{(-1)^k \Gamma(k+\frac{3}{2}) \kappa \sqrt{2}}{4 \pi \beta^{k+\frac{3}{2}}}
\end{equation}
with $\kappa$ chosen so that 
$$\widetilde \hol \gamma \barwedge \widetilde \hol \sigma = \kappa (c-1)^k+O\big((c-1)^{k+1}\big).$$
Let $\bu^u_c$ denote the continuous choice of unstable unit eigenvectors of $\rho_c(g)$ defined in a neighborhood of $c=1$,
which extends the choice of $\bu^u_1$ used in the definition of $\mu^u$.
Then we can define the real-analytic function $f(c)$ so that
$$\widetilde \hol \gamma \barwedge \widetilde \hol \sigma_l = f(c)\, (\bu^u_c \wedge \widetilde \hol \sigma_l).$$
Here $f(c)=\pm \|P_c^u \widetilde \hol \gamma\|$ and may be changing signs as $c$ passes through $1$. To compute $\kappa$ we just need to look at the lowest order terms in this expression computed at the point $c=1$. Observe that $\hol_1 \sigma_l \in \Z^2$
while $\bu^u_1$ has quadratic irrational slope so we know that $\bu^u_1 \wedge \hol_1 \sigma_l$ is a non-zero constant.
On the other hand using our definition of $k$ we know that $\widetilde \hol \gamma \barwedge \widetilde \hol \sigma_l$ vanishes to order $k-1$ and so
we must have $f = \alpha (c-1)^k + O((c-1)^{k+1})$ and with this definition of $\alpha$ we have
$\kappa = \alpha (\bu^u_1 \wedge \hol_1 \sigma_l) = \alpha \mu^u(\sigma_l)$ so we see by plugging in to \eqref{eq:lim} that
$$
\lim_{n \to +\infty} \frac{n^{k+\frac{3}{2}}}{|\lambda^u_1|^n} \big(\phi^n(\gamma) \cap \sigma_l\big) = \frac{(-1)^k \Gamma(k+\frac{3}{2}) \alpha \sqrt{2}}{4 \pi \beta^{k+\frac{3}{2}}} \mu^u(\sigma_l).$$
The constant in front of $\mu^u(\sigma_l)$ is now independent of $l$ so we see that 
\begin{equation}
\label{eq:Thm 4 constant}
\lim_{n \to +\infty} \frac{n^{k+\frac{3}{2}}}{(\lambda^u_1)^n} \cap_\ast \circ \phi^n(\gamma) = \frac{(-1)^k \Gamma(k+\frac{3}{2}) \alpha \sqrt{2}}{4 \pi \beta^{k+\frac{3}{2}}} \mu^u.
\end{equation}
\end{proof}

\subsection{Asymptotics of cylinder intersections}
\label{sect:intersection}

The goal of this section is to prove Theorem \ref{thm:mixing} and Corollary \ref{cor: dissipative}
of the introduction.

We first need to make a general remark about cylinders intersection on a translation surface.
Let $\sC$ be a cylinder on a translation surface $\bS$. A {\em core curve} $\gamma_\sC$ of $\sC$ is a closed geodesic in $\sC$. Such closed geodesic must wind once around the circumference of $\sC$. 

\begin{proposition}
\label{prop:intersection_formula}
Let $\sA$ and $\sB$ be cylinders on a translation surface $\bS$ with core curves $\gamma_\sA$ and $\gamma_\sB$. 
Assuming $\sA$ and $\sB$ are not parallel (i.e., $\hol \gamma_\sA \wedge \hol \gamma_\sB \neq 0$), we have
$$\textit{Area}({\mathcal A} \cap {\mathcal B})=
\frac{\big|\gamma_{\mathcal A} \cap \gamma_{\mathcal B}\big|\textit{Area}({\mathcal A}) \textit{Area}({\mathcal B})}
{\big|\hol(\gamma_{\mathcal A}) \wedge \hol(\gamma_{\mathcal B})\big|},
$$
where $\gamma_{\mathcal A} \cap \gamma_{\mathcal B}$ denotes algebraic intersection
number.
\end{proposition}

\begin{figure}[htbp]
\begin{center}
\includegraphics[width=3in]{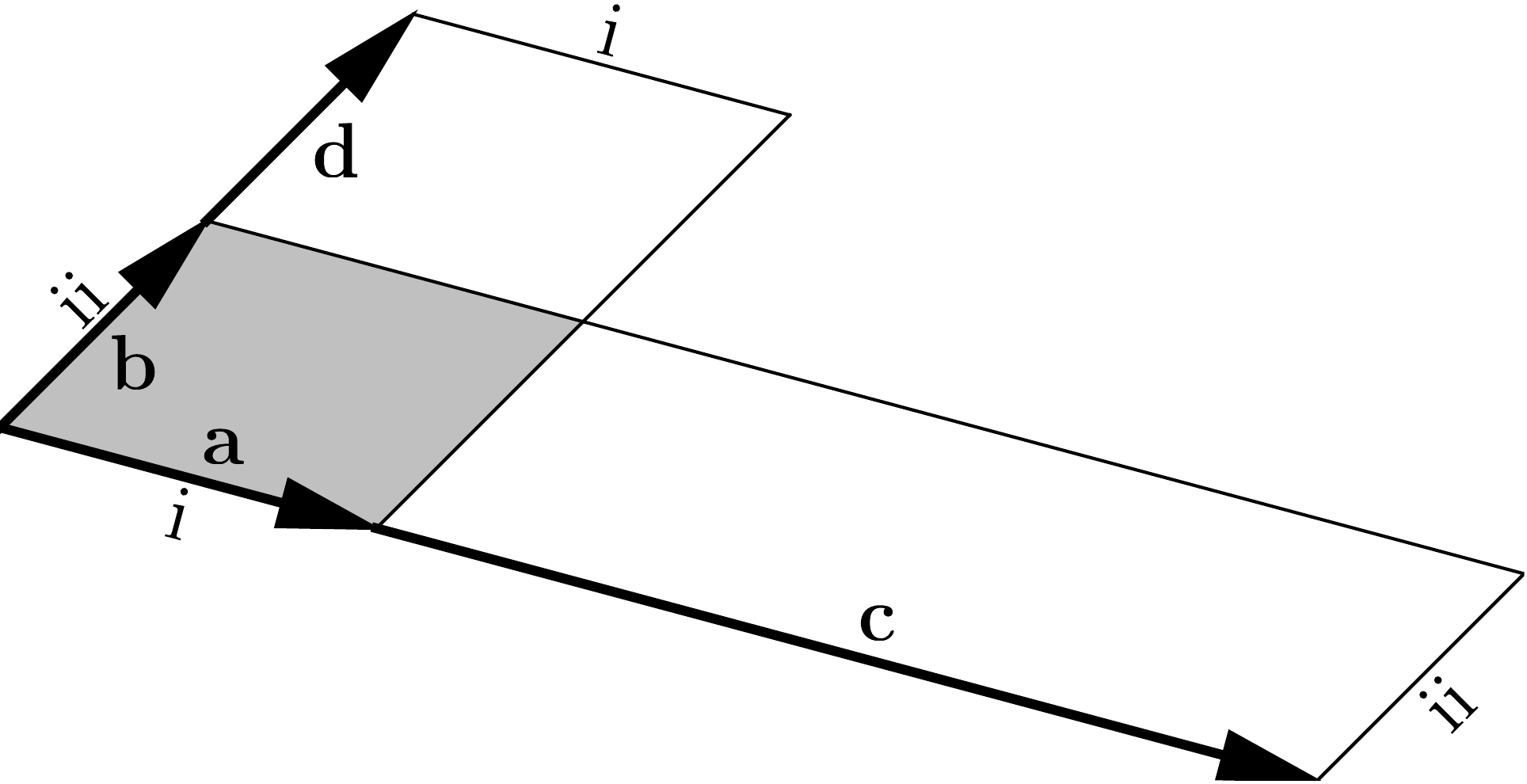}
\caption{Two intersecting cylinders developed into the plane. Roman numerals indicate edge identifications, which reconstruct the cylinders.}
\label{fig:cylinder_intersection}
\end{center}
\end{figure}

\begin{proof}
Cylinders on a translation surface intersect in a union of parallelograms, which are isometric and differ only by parallel translation. The number of these parallelograms is the absolute value of the algebraic intersection number between the core curves. Thus, we need to show that the area of one such parallelogram is given by 
\begin{equation}
\label{eq:parallelagram}
\frac{\textit{Area}({\mathcal A}) \textit{Area}({\mathcal B})}
{\big|\hol(\gamma_{\mathcal A}) \wedge \hol(\gamma_{\mathcal B})\big|}.
\end{equation}
Develop the two cylinders into the plane from an intersection as in Figure \ref{fig:cylinder_intersection}. Define the vectors ${\bf a}$, ${\bf b}$, ${\bf c}$, and ${\bf d}$ as in the figure. Then the area of the two cylinders is given by the quantities
$$\textit{Area}({\mathcal A})=|\ba \wedge (\bb+\bd)| \quad \text{ and} \quad \textit{Area}({\mathcal B})=|(\ba+\bc) \wedge \bb|.$$
Since the pair of vectors $\ba$ and $\bc$ are parallel as are the pair $\bb$ and $\bd$, we may write the product of areas as
$$\textit{Area}({\mathcal A})\textit{Area}({\mathcal B})=
|\ba \wedge \bb| |(\ba+\bc) \wedge (\bb+\bd)|.$$
The wedge of the holonomies may be written as
$$\big|\hol(\gamma_{\mathcal A}) \wedge \hol(\gamma_{\mathcal B})\big|=|(\ba+\bc) \wedge (\bb+\bd)|.$$
Thus, the quotient given in \eqref{eq:parallelagram} is $|\ba \wedge \bb|$, the area of the 
parallelogram formed by the intersection.
\end{proof}

\begin{proof}[Proof of Theorem \ref{thm:mixing}]
 Let $\sA$ and $\sB$ be cylinders on $\bP_1$ and let 
$\gamma_\sA$ and $\gamma_\sB$ be their core curves and let $\v, \bw \in \Z^2$ be their holonomies.
Let $\phi$ be a hyperbolic automorphism of $\bP_1$.
Then $D \phi=\rho_1(g)$ is a hyperbolic element of $\SL^\pm(2,\Z)$. As a consequence the eigenvalues of $D \phi$ are quadratic irrationals and so $(P_1^u \v) \wedge \bw$ is non-zero. In the context of Theorem \ref{thm:parabolic intersections} this implies that $k=0$ and $\kappa = (P_1^u \v) \wedge \bw$ and so
the theorem gives that
$$\phi^n(\gamma_\sA) \cap \gamma_\sB ~\sim~ (\lambda_1^u)^n n^{\frac{-3}{2}} \frac{\Gamma(\frac{3}{2})\kappa}{\pi \beta^\frac{3}{2} \sqrt{2}}~=~
 (\lambda_1^u)^n n^{\frac{-3}{2}} \frac{\kappa}{4 \beta^\frac{3}{2} \sqrt{2\pi}}.
$$
For sufficiently large values of $n$, the cylinders $\phi^n({\mathcal A})$ and
$\mathcal B$ are not parallel, so we may apply Proposition \ref{prop:intersection_formula} 
to obtain that
$$\textit{Area}\big(\phi^n({\mathcal A}) \cap {\mathcal B}\big)~=~
\frac{\big|\phi^n(\gamma_{\mathcal A}) \cap \gamma_{\mathcal B}\big|\textit{Area}({\mathcal A}) \textit{Area}({\mathcal B})}{|\rho_1(g)^n \v \wedge \bw|}
~\sim~ 
\frac{n^{\frac{-3}{2}} |\kappa| \textit{Area}({\mathcal A}) \textit{Area}({\mathcal B})}{4 \beta^\frac{3}{2} \sqrt{2\pi} |(\lambda_1^u)^{-n} \rho_1(g)^n \v \wedge \bw|}.
$$
The quantity $|(\lambda_1^u)^{-n} \rho_1(g)^n \v \wedge \bw|$ converges to $|\kappa|$, so this gives us the expression in the statement of the theorem.
\end{proof}

\begin{proof}[Proof of Corollary \ref{cor: dissipative}]
Fix $\phi:\bP_1 \to \bP_1$ hyperbolic.
We will explain how to produce the wandering sets $W_{i,k}$ so that $\bP_1 \smallsetminus \bigcup_{i,k} W_{i,k}$ has zero measure.
Let $\sA_i$ be the $i$-th horizontal cylinder as depicted in Figure \ref{fig:horizontal}.
By Theorem \ref{thm:mixing} applied to $\phi^{-1}$ we have
$$\sum_{n=0}^\infty \Area\big(\phi^{-n}(\sA_i) \cap \sA_i\big) < \infty.$$
This means that the set of points in $\sA_i$ which return infinitely often to $\sA_i$ has measure zero. For $k \geq 0$, let $W_{i,k}$ to be the set of points in $\sA_i$ which return exactly $k$ times to $\sA_i$ under the action of $\phi^{-1}$. From the remarks above we see that $\sA_i \smallsetminus \bigcup_{k \geq 0} W_{i,k}$ is measure zero.
Since $\bP_1=\bigcup_i \sA_i$ we see that $\sW=\{W_{i,k}\}$ satisfies the statements in Corollary \ref{cor: dissipative}.
\end{proof}

\subsection{Geometric intersection numbers}
\label{sect:geometric}
The space $\bP_1$ is non-positively curved in the sense that its universal cover is $CAT(0)$.
This guarantees that all homotopically non-trivial closed curves have geodesic representatives in the metric sense.
In particular every such curve has a realization as a sequence of saddle connections $\sigma_1, \ldots, \sigma_k$ so that the endpoint singularity of $\sigma_i$
coincides with the start singularity of $\sigma_{i+1 \pmod{k}}$ and the angle made between the two saddle connections at this singularity is at least $\pi$. A closed metric geodesic may also be a closed non-singular straight-line trajectory, but by moving to the boundary of the corresponding cylinder we can find a metric geodesic representative of this homotopy class consisting of a sequence of saddle connections.

We will now describe a way to compute the geometric intersection numbers between two non-trivial homotopy classes of closed curves in the punctured surface $\bP_1^\circ$.
First we may find metric geodesic representatives $\alpha=\alpha_1 \cup \ldots \cup \alpha_k$ and $\gamma=\gamma_1 \cup \ldots \cup \gamma_l$ for the curves in $\bP_1$ where the $\alpha_i$ and $\gamma_j$ are saddle connections. There are two types of intersections
between $\alpha$ and $\gamma$: those that occur at singularities and those that do not. The unit tangent bundle space $T_1 s_\ast$ at an infinite cone singularity $s_\ast$ is naturally homeomorphic to a line and we can make it a metric line using angle coordinates, identifying it with the universal cover of $\R/2\pi \Z$. Two saddle connections meeting at a singularity $s_\ast$ thus determine an interval $I$ in $T_1 s_\ast$. Our metric geodesics $\alpha$ and $\gamma$ thus determine two sequences of intervals
$I_1, \ldots, I_k$ and $J_1, \ldots J_l$ in the pair of lines $T_1 s_0 \cup T_1 s_1$. We say two intervals $I$ and $J$ are {\em linked} if $I$ contains an endpoint of $J$ and $J$ contains an endpoint of $I$.

\begin{proposition}
\label{prop:geometric intersection number}
Assume the metric geodesics $\alpha=\alpha_1 \cup \ldots \cup \alpha_k$ and $\gamma=\gamma_1 \cup \ldots \cup \gamma_l$ are transverse in $\bP^\circ_1$
(i.e., none of the saddle connections coincide). 
Then the geometric intersection number $i(\alpha, \gamma)$ is given by the sum of the number of intersections in $\bP^\circ_1$ and the number of linked pairs
of intervals $(I,J)$ with $I \in \{I_i\}$ and $J \in \{J_j\}$. 
\end{proposition} 
\begin{proof}
We will show that we can find representatives of the classes $\alpha$ and $\gamma$ which realize the claimed intersection number. Then we will argue that the representatives realize the intersection number.

There is an $r>0$ so that the open balls $B_r(s_0)$ and $B_r(s_1)$ do not intersect and only intersect the saddle connections in the set $\{\alpha_i\} \cup \{\gamma_j\}$
in segments of length $r$ where the saddle connections enter and exit the singularities. Let $\H$ be the hyperbolic upper half plane. Because the singularities are infinite cone singularities, the balls are homeomorphic to $\H \cup \{\infty\}$ via a
sending radial lines in the ball to vertical lines in $\H$ and sending the singularity to $\infty$. Consecutive arcs of saddle connections of $\alpha$ and $\gamma$ intersecting the balls $B_r(s_\ast)$ are then sent to a vertical geodesic headed to up infinity followed by a vertical geodesic back down to the real line in $\partial \H$. We can straighten such an arcs to a geodesic in $\H$ joining the places where the two arcs pass through the real axis. Do this for all the visits of the saddle connections to the singularities. Observe that two geodesics with distinct endpoints in $\H$ joining points in $\R$ intersect if and only if the corresponding intervals in $\R$ link. Transversality guarantees this distinct endpoint condition for the geodesics constructed as above. Thus performing this action results in a pair of curves $\hat \alpha$ and $\hat \gamma$ that intersect in precisely the number of times stated in the theorem. 

We must argue that the number of intersections between $\hat \alpha$ and $\hat \gamma$ is minimal among curves in their homotopy classes. For this it suffices to show
that we can not find a simple closed curve formed from an arc of $\hat \alpha$ and an arc of $\hat \gamma$ which is homotopically trivial; see 
\cite[Proposition 3.10]{FLP}. Such arcs must join together at a pair of intersections for $\hat \alpha$ and $\hat \gamma$.
Existence of such a curve is ruled out by the Gauss-Bonnet theorem which promises that for a closed polygon in the plane, the total exterior angle is $2\pi$. Indeed suppose we had two such arcs of $\hat \alpha$ and $\hat \gamma$ which formed a simple closed homotopically trivial loop $\hat \eta$. Then $\hat \eta$ bounds a topological disk.
By deforming back to the flat geodesics $\alpha$ and $\gamma$ 
we get a curve $\eta$ that bounds a flat polygonal disk. By transversality, the interior angles at the intersections are positive, so the exterior angles at these intersection points are each strictly less than $\pi$. The other vertices of $\eta$ must come from visits of $\alpha$ or $\gamma$ to the singularities, and since they are metric geodesics the interior angles are at least $\pi$ and so the exterior angles are negative. We conclude that the total exterior angle of $\eta$ is less than $2 \pi$, which contradicts the existence of $\eta$.
\end{proof}

Lemma \ref{lem:intersection parabola 2} tells us how to estimate (within $1$) the number of interior intersections of saddle connections. Combining this lemma with the above result yields:
\begin{proposition}
\label{prop:geometric intersections}
Let $\alpha=\alpha_1 \cup \ldots \cup \alpha_k$ and $\gamma=\gamma_1 \cup \ldots \cup \gamma_l$ be transverse closed metric geodesics in $\bP_1$. Then the geometric intersection number
$i(\alpha,\gamma)$ is within $2 k l$ of
$$\sum_{i=1}^k \sum_{j=1}^l |\alpha_i \cap \gamma_j| = \sum_{i=1}^k \sum_{j=1}^l \left| \frac{1}{2\pi} \int_0^\pi \big((\widetilde \hol~\alpha_i) \wedge (\widetilde \hol~\gamma_j) \big) (1-\cos t)~dt \right|.$$
\end{proposition}
Note that the identity in the equation above holds by definition see \eqref{eq:extended intersection number}.
\begin{proof}
We get an error of up to $kl$ from possible linked intervals in Proposition \ref{prop:geometric intersection number} and another error of up to one from comparing
each integral to the number of interior intersections of the corresponding saddle connections (see Lemma \ref{lem:intersection parabola 2}).
\end{proof}

Now consider an affine automorphism $\phi:\bP_1 \to \bP_1$. Let $\alpha=\alpha_1 \cup \ldots \cup \alpha_k$ be a closed metric geodesic on $\bP_1$. Observe that $\phi(\alpha)=\phi(\alpha_1) \cup \ldots \cup \phi(\alpha_k)$ is also a closed metric geodesic because this image satisfies the angle condition on consecutive saddle connections. As a consequence we see that when considering the geometric intersection number $i\big(\phi^n(\alpha),\gamma\big)$ we can use the integral formula above and the only effect is that we introduce a uniformly bounded error.

Fix $\phi:\bP_1 \to \bP_1$ with hyperbolic derivative $D \phi$. Let $\bu^u$ and $\bu^s$ be unstable and stable unit eigenvectors for $D\phi$. Then the stable and unstable 
elements $\mu^u, \mu^s \in \R^{\sS(\bP_1^\circ)}$ corresponding to the transverse measures on $\bP_1$ to foliations parallel to the expanding and contracting 
directions are defined by evaluating them on a simple closed curve with metric geodesic representative $\alpha=\alpha_1 \cup \ldots \cup \alpha_k$ by 
\begin{equation}
\label{eq:stable and unstable}
\mu^u(\alpha) = \sum_{i=1}^k |\bu^u \wedge \hol_1 \alpha_i|
\quad \text{and}\quad 
\mu^s(\alpha) = \sum_{i=1}^k |\bu^s \wedge \hol_1 \alpha_i|.
\end{equation}

\begin{proof}[Proof of Theorem \ref{thm:geometric}]
Statement (1) is standard: In the unstable direction we have
$$
\begin{array}{rcl}
\mu^u \circ \phi^{-1}(\alpha) & = & \sum_{i=1}^k |\bu^u \wedge D(\phi)^{-1}(\hol_1 \alpha_i)|=
\sum_{i=1}^k | D(\phi)(\bu^u) \wedge \hol_1 \alpha_i|
\\ &= &\sum_{i=1}^k |\lambda^u_1 \bu^u \wedge \hol_1 \alpha_i|=|\lambda^u_1| \mu^u(\alpha).
\end{array}$$
A similar argument works in the stable direction.

Now fix a homotopically non-trivial simple closed curve and let $\alpha=\alpha_1 \cup \ldots \cup \alpha_k$ be a metric geodesic representative.
We will prove that $\frac{n^\frac{3}{2}}{|\lambda_1^u|^n} i_\ast \circ \phi^n(\alpha)$ converges to a constant times $\mu^s$ with the constant as given in the theorem.
To prove this let $\gamma$ be another homotopically non-trivial simple closed curve and let $\gamma=\gamma_1 \cup \ldots \cup \gamma_k$ be a metric geodesic representative.
Fixing an $i$ and a $j$, and observe that $\hol_1 \alpha_i$ and $\hol_1 \gamma_j$ are non-zero vectors in $\Z^2$ so that $(P^1_u \hol_1\alpha_i) \wedge \hol_1 \gamma_k \neq 0$.
This tells us that $k=0$ in Theorem \ref{thm:parabolic intersections} and 
$$|\kappa|=|(P^1_u \hol_1\alpha_i) \wedge \hol_1 \gamma_k|=
\left|\Big(\frac{\bu^s \wedge \hol_1\alpha_i}{\bu^s \wedge \bu^u} \bu^u\Big) \wedge \hol_1 \gamma_k\right|=
\frac{\mu^s(\alpha_i) \mu^u(\gamma_j)}{|\bu^u \wedge \bu^s|},$$
where we are slightly abusing notation by applying \eqref{eq:stable and unstable} to saddle connections (but this is justified if we think of these functions as determining
measured foliations.) Then from Theorem \ref{thm:parabolic intersections} we get
$$\lim_{n \to \infty} \frac{n^\frac{3}{2}}{|\lambda_1^u|^n} |\phi^{-n}(\alpha_i) \cap \gamma_j|=
\frac{\mu^s(\alpha_i) \mu^u(\gamma_j)}{4 \beta^{\frac{3}{2}} \sqrt{2\pi}\,|\bu^u \wedge \bu^s|}.$$
Then it follows from Proposition \ref{prop:geometric intersections} that
$$\lim_{n \to \infty} \frac{n^\frac{3}{2}}{|\lambda_1^u|^n} i\big(\phi^{n}(\alpha),\gamma\big)=\sum_{i=1}^k \sum_{j=1}^l \frac{\mu^s(\alpha_i) \mu^u(\gamma_j)}{4 \beta^{\frac{3}{2}} \sqrt{2\pi}\,|\bu^u \wedge \bu^s|}
=\frac{\mu^s(\alpha) \mu^u(\gamma)}{4 \beta^{\frac{3}{2}} \sqrt{2\pi}\,|\bu^u \wedge \bu^s|}.$$
This is equivalent to the first limiting statement in statement (2) of the theorem. The second limit can be obtained by switching $\phi$ for $\phi^{-1}$ and stable for unstable.
\end{proof}

\section*{Software used}
Software was used in several ways in this paper. SageMath \cite{sagemath18} and the FlatSurf SageMath Module \cite{flatsurf} were used to experimentally check the results in this paper. FlatSurf was also used to generate the figures of translation surfaces.

\section*{Acknowledgments}
Theorem \ref{thm:mixing} was first proved while the author was a postdoc at Northwestern and the author would like to thank John Franks and Amie Wilkinson for helpful conversations at the time.
The author would also like to thank Barak Weiss for some more recent conversations,
and the anonymous referee for suggesting a number of improvements.
This article
is based upon work supported by the National Science Foundation under
Grant Number DMS-1500965 as well as a PSC-CUNY Award (funded by The Professional Staff
Congress and The City University of New York).

\bibliographystyle{amsalpha}
\bibliography{/home/pat/ownCloud/math/my_papers/bibliography}
\end{document}